\documentclass[10pt]{article}

\usepackage[small, pagestyles]{titlesec}
\usepackage[small, it]{caption}
\usepackage{enumitem}

\usepackage{calc}

\usepackage{xspace}

\usepackage[T1]{fontenc}

\usepackage{amsmath}
\usepackage{amsthm}
\usepackage{amssymb}
\usepackage{latexsym}

\usepackage[square, comma, numbers, sort&compress]{natbib}

\usepackage{cancel}

\usepackage{color}
\definecolor{lightgray}{rgb}{0.8, 0.8, 0.8}
\definecolor{darkgray}{rgb}{0.7, 0.7, 0.7}
\definecolor{darkblue}{rgb}{0, 0, .4}

\usepackage[bookmarks]{hyperref}
\hypersetup{
	colorlinks=true,
	linkcolor=black,
	anchorcolor=black,
	citecolor=black,
	urlcolor=black,
	pdfpagemode=UseThumbs,
	pdftitle={Uncountably many enumerations of well-quasi-ordered permutation classes},
	pdfsubject={Combinatorics},
	pdfauthor={Brignall and Vatter},
}

\newcommand{\minisec}[1]{\medskip\noindent{\bf #1.}}

\theoremstyle{plain}
\newtheorem{theorem}{Theorem}[section]
\newtheorem{proposition}[theorem]{Proposition}
\newtheorem*{proposition*}{Proposition}

\newtheorem{lemma}[theorem]{Lemma}

\newtheorem{observation}[theorem]{Observation}
\newtheorem{falseconjecture}[theorem]{False Conjecture}

\newtheorem{question}[theorem]{Question}

\newtheorem*{remark*}{Remark}

\renewenvironment{abstract}%
	{%
		\begin{list}{}%
			{%
				\setlength{\rightmargin}{1in}%
				\setlength{\leftmargin}{1in}%
			}%
			\item[]\ignorespaces
			\begin{small}
		}%
		{%
			\end{small}
			\unskip
		\end{list}
	}

\newcommand{\Av}{\operatorname{Av}}

\newcommand{\C}{\mathcal{C}}

\renewcommand{\L}{\mathcal{L}}

\renewcommand{\P}{\mathcal{P}}

\renewcommand{\P}{\mathcal{P}}

\newcommand{\gr}{\mathrm{gr}}

\newcommand{\ugr}{\overline{\gr}}

\newcommand{\exparen}[1]{{\scriptscriptstyle(}\!\!\;#1\!\!\;{\scriptscriptstyle)}}
\usepackage{tikz}
\newcommand{\plotptradius}{4pt}

\tikzset{point/.style={circle, draw, fill=black, inner sep=0pt, minimum width=\plotptradius}}
\tikzset{openpoint/.style={circle, draw, fill=none, inner sep=0pt, minimum width=\plotptradius}}
\tikzset{empty/.style={circle, draw=none, fill=none, minimum width=0}}

\renewcommand{\d}{\mathsf{d}}
\renewcommand{\u}{\mathsf{u}}
\renewcommand{\r}{\mathsf{r}}
\newcommand{\rd}{\mathsf{r_d}}
\newcommand{\ru}{\mathsf{r_u}}
\newcommand{\alphabet}{\{\rd,\ru,\d,\u\}}

\newcommand{\norigin}{\circ}
\newcommand{\origin}{\bullet}

\usepackage[multiple, flushmargin, hang]{footmisc}

\newpagestyle{main}[\small]{
        \headrule
        \sethead[\usepage][][]
        {\sc Uncountably Many Enumerations of WQO Permutation Classes}{}{\usepage}}

\title{\sc Uncountably Many Enumerations of Well-Quasi-Ordered Permutation Classes}

\author{\centering
	\begin{tabular}{ccc}
	Robert Brignall
	&\rule{0pt}{0pt}&
	Vincent Vatter%
	\footnote{Vatter's research was partially supported by the Simons Foundation via award number 636113.}
	\\[-0.25ex]
	\small School of Mathematics and Statistics
	&&
	\small Department of Mathematics\\[-0.5ex]
	\small The Open University
	&&
	\small University of Florida\\[-0.5ex]
	\small Milton Keynes, England UK
	&&
	\small Gainesville, Florida USA\\[-1.5ex]
	\end{tabular}
	\vspace{0.3in}
}

\titleformat{\section}
        {\large\sc}
        {\thesection.}{1em}{}

\begin{document}
\maketitle

\vspace{-0.25in}

\begin{abstract}
We construct an uncountable family of well-quasi-ordered permutation classes, each with a distinct enumeration sequence. This disproves a conjecture that all well-quasi-ordered permutation classes have algebraic generating functions, and in fact shows that many such classes lack D-finite or D-algebraic generating functions. Our construction is based on an uncountably large collection of factor-closed, well-quasi-ordered binary languages due to Pouzet.
\end{abstract}

\pagestyle{main}

\section{Introduction}

Well-quasi-order is a fundamental structural property of permutation classes. Until recently, it was conceivable that only countably many classes could possess this property, but recent work of Oudrar, Pouzet, and Zaguia~\cite{oudrar:minimal-prime-ages:} established that there are in fact uncountably many well-quasi-ordered permutation classes.
Our main result goes further.
We show that one can find uncountably many well-quasi-ordered permutation classes with distinct enumerations, which implies that these enumerations cannot all be algebraic, D-finite, or otherwise well-behaved. We present our main result first, followed by definitions and context.

\begin{theorem}
\label{thm-wqo-not-algebraic}
There are uncountably many distinct enumerations of well-quasi-ordered permutation classes.
\end{theorem}

Our proof of Theorem~\ref{thm-wqo-not-algebraic} builds on a construction by Pouzet of factor-closed sets of binary words, adapted here to the setting of permutation classes.

We now provide the necessary definitions and context from the study of permutation classes. We think of permutations in one-line notation, so a \emph{permutation of length $n$} is simply an ordering of the set $\{1,2,\dots,n\}$. We work with the classical, non-consecutive, permutation patterns, so the permutation~$\pi$ \emph{contains} the permutation~$\sigma$ of length $k$ if~$\pi$ has a subsequence of length $k$ that is \emph{order isomorphic} to~$\sigma$, by which we mean that the subsequence has the same pairwise comparisons as~$\sigma$. If $\pi$ contains $\sigma$, we sometimes say that~$\sigma$ is a \emph{subpermutation} of~$\pi$.

For example,~$\pi=372694185$ contains~$\sigma=32514$, as witnessed by its subsequence $32918$, but~$\pi$ avoids $54321$ because it has no decreasing subsequence of length five. This containment is illustrated in Figure~\ref{fig-pattern-containment}, where $\pi$ is plotted as a function.

\begin{figure}
\begin{center}
\begin{tikzpicture}[scale=0.3]
	\begin{scope}[shift={(0,0)}]
		\draw (0,0) rectangle (10,10);
		
		\node[point] (b1) at (1,3) {};
		\node[point] (b2) at (2,7) {};
		\node[point] (b3) at (3,2) {};
		\node[point] (b4) at (4,6) {};
		\node[point] (b5) at (5,9) {};
		\node[point] (b6) at (6,4) {};
		\node[point] (b7) at (7,1) {};
		\node[point] (b8) at (8,8) {};
		\node[point] (b9) at (9,5) {};
		
		\draw (b1) circle [radius=0.6];
		\draw (b3) circle [radius=0.6];
		\draw (b5) circle [radius=0.6];
		\draw (b7) circle [radius=0.6];
		\draw (b8) circle [radius=0.6];
		
		\node at (1,0) [below] {$3$};
		\node at (2,0) [below] {$7$};
		\node at (3,0) [below] {$2$};
		\node at (4,0) [below] {$6$};
		\node at (5,0) [below] {$9$};
		\node at (6,0) [below] {$4$};
		\node at (7,0) [below] {$1$};
		\node at (8,0) [below] {$8$};
		\node at (9,0) [below] {$5$};
	\end{scope}

	\node at (15,5) {\Large$\ge$};

	\begin{scope}[shift={(20,0)}]
		\draw (0,0) rectangle (10,10);
		
		\node[point] (a1) at (1,3) {};
		\node[point] (a2) at (3,2) {};
		\node[point] (a3) at (5,9) {};
		\node[point] (a4) at (7,1) {};
		\node[point] (a5) at (8,8) {};
		
		\node at (1,0) [below] {$3$};
		\node at (3,0) [below] {$2$};
		\node at (5,0) [below] {$9$};
		\node at (7,0) [below] {$1$};
		\node at (8,0) [below] {$8$};
	\end{scope}

\end{tikzpicture}
\end{center}
\caption{The permutation $372694185$ contains the permutation $32514$.}
\label{fig-pattern-containment}
\end{figure}

A \emph{permutation class} is a downset of permutations in this order. In other words, if $\C$ is a permutation class,~$\pi\in\C$, and~$\sigma\le\pi$, then we must also have~$\sigma\in\C$. Every permutation class can be described by a unique \emph{basis}, which consists of the minimal permutations \emph{not} in the class. The basis of a class is necessarily an \emph{antichain}, meaning that none of its members is contained in another. As there are infinite antichains of permutations, there are infinitely-based permutation classes. Given an antichain $B$ of permutations, the class with $B$ as its basis is denoted by $\Av(B)$. A permutation class is \emph{well-quasi-ordered} (\emph{wqo} for short, or \emph{belordonn{\'e}} in French) if it does not contain an infinite antichain.

With $\C_n$ denoting the set of permutations of length $n$ in the class $\C$ and $|\pi|$ denoting the length of a permutation~$\pi$, the generating function of $\C$ is
\[
	\sum_{n\ge 0} |\C_n|x^n=\sum_{\pi\in\C} x^{|\pi|}.
\]
We are often interested in whether a generating function~$f(x)$ is \emph{rational} (if~$f(x)=p(x)/q(x)$ for polynomials $p$ and $q$), \emph{algebraic} (if there is a non-zero polynomial~$p(x,y)\in\mathbb{Q}[x,y]$ such that $p(x,f(x))=0$), or \emph{D-finite} (if $f(x)$ and its derivatives span a finite dimensional vector space over~$\mathbb{Q}(x)$).

In their 1996 paper, Noonan and Zeilberger~\cite{noonan:the-enumeration:} conjectured that every finitely-based permutation class has a D-finite generating function. Note that the finite basis hypothesis is clearly necessary: because there are infinite antichains of permutations, there are uncountably many different enumerations of permutation classes, but only countably many D-finite generating functions with rational coefficients. The Noonan--Zeilberger conjecture had begun to seem unlikely, and Zeilberger himself conjectured that it was false in 2005~\cite{elder:problems-and-co:}. The conjecture was finally disproved by Garrabrant and Pak in 2015~\cite{garrabrant:permutation-pat:}.%
\footnote{The proof of Garrabrant and Pak does not provide a concrete counterexample to the Noonan--Zeilberger conjecture. Many believe that the class of $1324$-avoiding permutations does not have a D-finite generating function because it is suspected to have the wrong asymptotics for such a generating function~\cite{conway:on-the-growth-r:,conway:1324-avoiding-p:}. The same type of analysis applies to several other classes avoiding a single permutation of length five, as shown by Clisby, Conway, Guttmann, and Inoue~\cite{clisby:classical-lengt:}. That said, for the sake of providing a concrete counterexample to the Noonan--Zeilberger conjecture, it may be preferable to consider classes with more basis elements, and thus more structure. Several potential counterexamples along these lines were identified in the work of Albert, Homberger, Pantone, Shar, and Vatter~\cite{albert:generating-perm:}. Others have later been identified by B\'ona and Pantone~\cite{bona:permutations-av:long-mono}.}
In an interview published in 2021~\cite{mansour:interview-with-:stanley}, Stanley singled out the results of Garrabrant and Pak, saying:
\begin{quote}
a topic within enumerative combinatorics that seems ripe for further investigation is developing a theory for showing that a given generating function (in one variable) does \emph{not} have some desirable property, such as being D-finite or differentially algebraic. There are a number of results in this area, but nothing approaching a general theory. The most significant work in this direction (in my opinion) is the disproof of the Noonan--Zeilberger conjecture by Scott Garrabrant and Igor Pak.
\end{quote}

Thus we know by elementary cardinality arguments that some permutation classes fail to have well-behaved enumerations, and the Garrabrant--Pak result shows that this fails even for some \emph{finitely-based} permutation classes. In a similar direction, at the same conference at which he conjectured that the Noonan--Zeilberger conjecture was false, Zeilberger asked for necessary and sufficient conditions for classes to have rational, algebraic, D-finite, or similarly structured generating functions~\cite{elder:problems-and-co:}. However, the results of Albert, Brignall, and Vatter~\cite{albert:large-infinite-:} indicate that this question is likely to be intractable, as having a rational generating function provides little structural information; classes with rational generating functions can be just as complex as arbitrary permutation classes.

\begin{theorem}[Albert, Brignall, and Vatter~\cite{albert:large-infinite-:}]
\label{thm-contained-rational}
Every permutation class except for the class of all permutations is contained in a class with a rational generating function.%
\footnote{While Theorem~\ref{thm-contained-rational} implies the Marcus--Tardos theorem (formerly the Stanley--Wilf conjecture, stating that all proper permutation classes grow at most exponentially), the proof of Theorem~\ref{thm-contained-rational} uses this result.}
\end{theorem}

\begin{figure}[t]
    \centering
    \setlength{\fboxrule}{0.4pt} 
    \setlength{\fboxsep}{0pt}    
    \fbox{\includegraphics[width=1.49in]{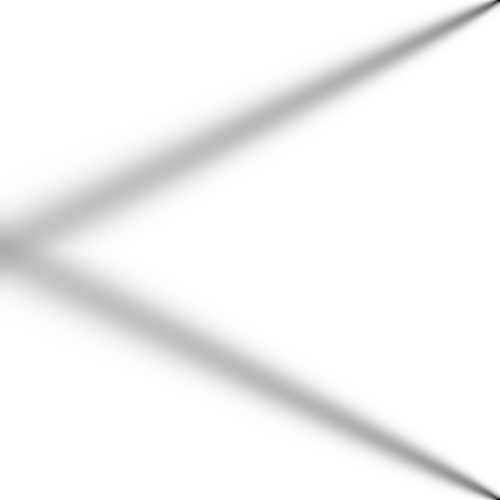}}
    \qquad
    \fbox{\includegraphics[width=1.49in]{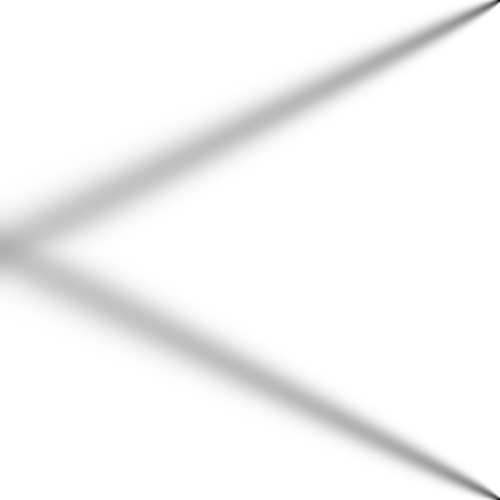}}
    \qquad
    \fbox{\includegraphics[width=1.49in]{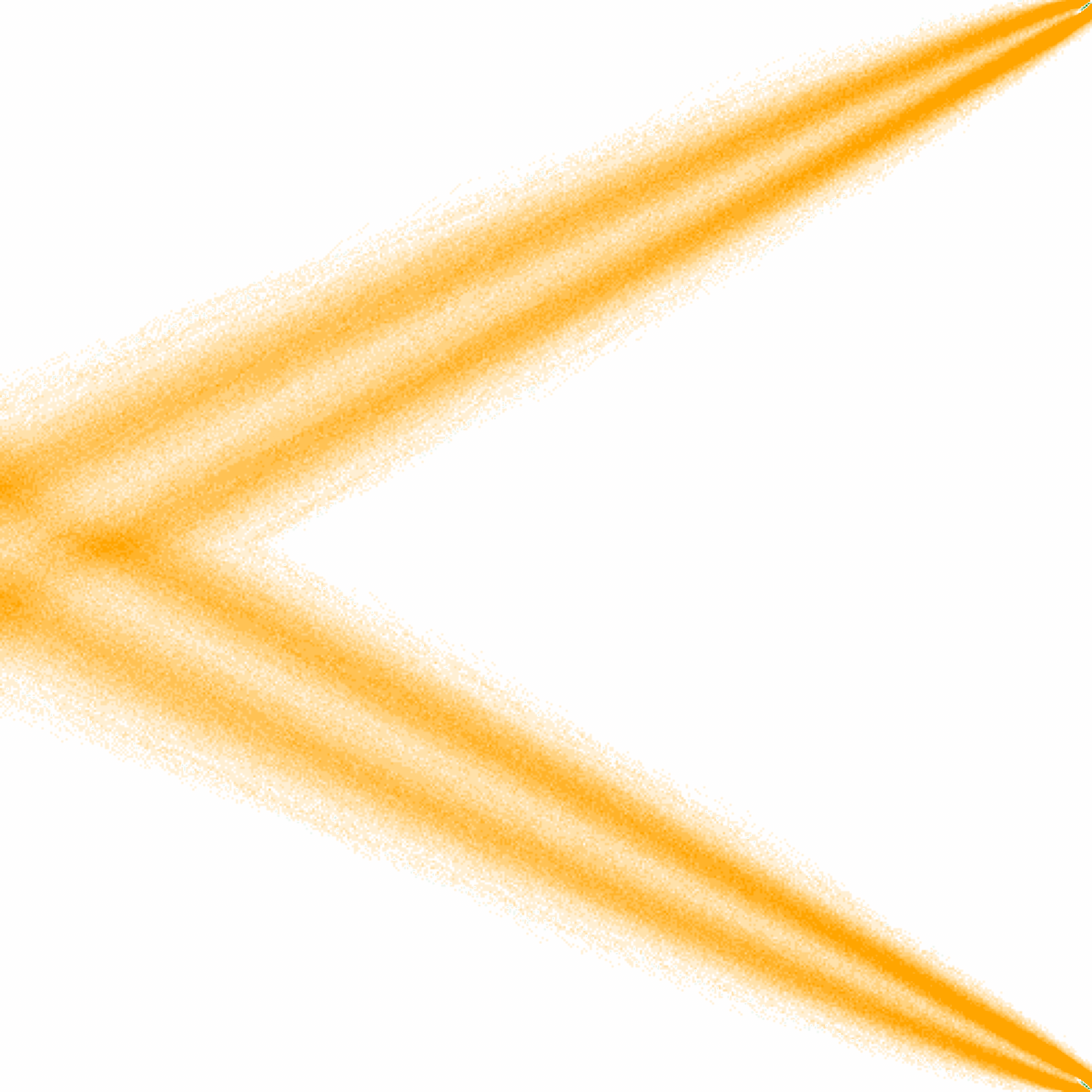}}
    \bigskip
	\caption{A heatmap of the class $\Av(1432, 4123)$ on the left, of its subclass $\Av(1432, 4123, 13452, 14523, 24513, 42153, 52143, 53214)$ in the middle, and an image of the difference of these two heatmaps on the right. All of the classes constructed here lie in both of these classes. The heatmaps were created by Jay Pantone by sampling one million permutations of length $500$ uniformly at random, using Combinatorial Exploration~\cite{albert:combinatorial-e:}.}
	\label{fig-heatmap-1432-4123}
\end{figure}

Indeed, the classes we construct to prove Theorem~\ref{thm-wqo-not-algebraic} all lie within $\Av(1432, 4123)$; this class not only has a rational generating function,%
\footnote{To see this, note that by reversing and then inverting the members of the class, we obtain the symmetric class $\Av(4123, 3214)$. Since every permutation of length at least~$5$ contains $123$ or $321$, there can be at most~$5$ places in which to insert a new maximum into a member of $\Av(4123, 3214)$. This implies not only  that this class has a regular insertion encoding~\cite{albert:the-insertion-e:,vatter:finding-regular:}, but that it even has a finitely labelled generating tree~\cite{vatter:finitely-labele:}; indeed, this generating tree was first computed in 2003 by Kremer and Shiu~\cite[Proposition~10]{kremer:finite-transiti:}. To automatically compute the generating function for such a class now, one would want to employ Combinatorial Exploration~\cite{albert:combinatorial-e:}.} but its enumeration is given by the remarkably simple formula~${(4^{n-1} + 2)/3}$. The heatmap for this class shown on the left of Figure~\ref{fig-heatmap-1432-4123} also speaks to its simple structure. All our construction requires of this class is that it admit an interpretation of the factor order on binary words, and in fact the classes we construct all lie in the downward closure of the set of rightward-yearning pin sequences (see Section~\ref{sec-rightward-yearning}). Theorem~\ref{thm-rightward-basis-enum}, due to Jarvis~\cite{jarvis:pin-classes:}, shows that this class is
\[
	\Av(1432, 4123, 13452, 14523, 24513, 42153, 52143, 53214),
\]
and also has a simple, rational, generating function. A heatmap for this class is shown in the center of Figure~\ref{fig-heatmap-1432-4123}. Indeed, the downward closure of the set of \emph{all} pin sequences was shown by Bassino, Bouvel, and Rossin~\cite{bassino:enumeration-of-:} to have a rational generating function as well.

We say that a permutation class is \emph{strongly algebraic} (or \emph{h{\'e}r{\'e}ditairement alg{\'e}brique} in French~\cite{oudrar:sur-lenumeratio:}) if it and all its subclasses have algebraic generating functions. By an elementary counting argument, it follows that strongly-algebraic permutation classes must be wqo. Thus it was natural to conjecture the converse, which appears to have been first stated in the literature by the second author.

\begin{falseconjecture}[Vatter~{\cite[Conjecture 12.3.4]{vatter:permutation-cla:}}]
\label{conj-wqo-algebraic}
A permutation class is strongly algebraic if and only if it is well-quasi-ordered.
\end{falseconjecture}

Theorem~\ref{thm-wqo-not-algebraic} shows that this conjecture is false, because there are too many wqo permutation classes with distinct enumerations for them all to have algebraic, D-finite, or otherwise well-behaved generating functions.%
\footnote{Note that the classes we construct in the proof of Theorem~\ref{thm-wqo-not-algebraic} are almost certainly infinitely based, and thus are unlikely to represent additional counterexamples to Noonan and Zeilberger's false conjecture.}

We remark that while our results show that wqo is not enough to guarantee algebraicity, a stronger notion known as \emph{labelled well-quasi-order} (\emph{lwqo} for short, or \emph{h{\'e}r{\'e}ditairement belordonn{\'e}} in French) might be. While this notion is implicit in the classic papers of Higman~\cite{higman:ordering-by-div:} and Kruskal~\cite{kruskal:well-quasi-orde:}, it was not until the work of Pouzet in the 1970s (in particular, his 1972 paper~\cite{pouzet:un-bel-ordre-da:}) that it was studied explicitly. It has been shown that lwqo nicely explains many structural results about permutation classes~\cite{brignall:labelled-well-q:}. Rather than risk another false conjecture, we raise the following as a question.

\begin{question}
\label{ques-lwqo-algebraic}
Does every lwqo permutation class have an algebraic generating function?
\end{question}

At the very least, Question~\ref{ques-lwqo-algebraic} will not fall to an argument such as the one we use here to disprove False Conjecture~\ref{conj-wqo-algebraic}, as there are only countably many lwqo permutation classes, which is a corollary of them all being finitely based, see~\cite[Proposition 2.3]{brignall:labelled-well-q:}. A slightly weaker version of Question~\ref{ques-lwqo-algebraic} has been conjectured by Oudrar~\cite[page 87]{oudrar:sur-lenumeratio:}.

Since lwqo is a stronger property than wqo and also ensures that classes are finitely based (although the converse does not hold; see Brignall, Engen, and Vatter~\cite{brignall:a-counterexampl:}), a positive answer to the following, strictly stronger question would resolve Question~\ref{ques-lwqo-algebraic}. Although we suspect the answer is negative, no counterexamples are currently known.

\begin{question}
\label{ques-wqo-fb-algebraic}
Does every finitely-based wqo permutation class have an algebraic generating function?
\end{question}

The remainder of this paper is organized as follows. In Section~\ref{sec-pouzet-factors}, we present Pouzet’s construction and the properties of the resulting (uncountably large) collection of factor-closed binary languages. Sections~\ref{sec-rightward-yearning} and~\ref{sec-construction} describe how these languages can be used to construct permutation classes, via pin sequences. We analyze the decomposition properties of the resulting classes in Sections~\ref{sec-decomp} and~\ref{sec-indecomposable}, and establish that each such class is well-quasi-ordered in Section~\ref{sec-wqo}. In Section~\ref{sec-enum}, we show that these classes have distinct enumeration sequences, completing the proof of our main result. The paper concludes with some final remarks in Section~\ref{sec-concluding}.

\section{Words Under the Factor Order}
\label{sec-pouzet-factors}
\newcommand{\alphabar}{\overline{\alpha}}

The construction described in this section is adapted from Pouzet's thesis~\cite[page~64]{pouzet:sur-la-theorie-:} though specialized to the binary alphabet $\{0,1\}$ here. We are interested in the \emph{factor order} on the set of all finite binary words, $\{0,1\}^\ast$, in which $u\le w$ if~$u$ occurs as a \emph{factor} (or \emph{consecutive subword}) of~$w$. In other words,~$u\le w$ in this order if~$w=w_1uw_2$ for possibly empty words~$w_1,w_2\in\{0,1\}^\ast$. This poset is \emph{not} wqo in general, as witnessed by the infinite antichain $11$, $101$, $1001$, $\dots$.
For a binary word $w=w(1)\cdots w(n)$, the \emph{complement} (or \emph{ones complement}) is the word formed by inverting every letter of~$w$, so
\[
	\overline{w}
	=
	\overline{w(1)}
	\cdots
	\overline{w(n)}
	=
	(1-w(1))\cdots (1-w(n)).
\]

We now describe Pouzet's construction. While it takes a bit of preparation, this will eventually give us uncountably many binary wqo languages, each closed downward in the factor order and with a distinct enumeration sequence.

Let $(s_k)_{k\in\mathbb{N}}$ be a sequence of positive integers. We construct an infinite sequence~$(\alpha^{\exparen{s_k}}_i)_{i\in\mathbb{N}}$ of binary words by setting
\[
	\alpha^{\exparen{s_k}}_1=01
\]
and, for $i\ge 1$, defining
\[
		\alpha^{\exparen{s_k}}_{i+1}
			=
			\underbrace{\alpha^{\exparen{s_k}}_{i}\cdots \alpha^{\exparen{s_k}}_{i}}_{\text{$s_{i}$ copies}}\,
			\underbrace{\alphabar^{\exparen{s_k}}_{i}\cdots \alphabar^{\exparen{s_k}}_{i}}_{\text{$s_{i}$ copies}}
			=
			\left(\alpha^{\exparen{s_k}}_{i}\right)^{s_i}
			\left(\alphabar^{\exparen{s_k}}_{i}\right)^{s_i}.
\]

In the case where $(s_k)=(1,1,1,\dots)$, the sequence~$\left(\alpha^{\exparen{s_k}}_i\right)$ consists of prefixes of the well-known Prouhet--Thue--Morse sequence (see~\cite{allouche:the-ubiquitous-:}):
\[
	\begin{array}{rl}
	i	&\alpha^{(1,1,1,\dots)}_i	\\[2pt]
	\hline\\[-10pt]
	1	&01						\\
	2	&01\,10					\\
	3	&0110\,1001				\\
	4	&01101001\,10010110	
	\end{array}
\]
This example illustrates two features of the construction that hold for arbitrary sequences~$(s_k)$. First, we can read the lengths of these words directly from the construction:
\[
	|\alpha^{\exparen{s_k}}_i|= 2^is_1s_2\cdots s_{i-1}.
\]
Second, the first letter of~$\alpha^{\exparen{s_k}}_i$ is always~$0$, while the last letter alternates with $i$.

\begin{observation}\label{obs-last-letters}
The last letter of~$\alpha^{\exparen{s_k}}_i$ is 0 if and only if $i$ is even.
\end{observation}



%

Of particular interest to us is the language
\[
	\L^{\exparen{s_k}}
	=
	\{w : \text{$w$ is a factor of~$\alpha^{\exparen{s_k}}_i$ for some $i\ge 1$}\}.
\]
By definition, we include the empty word $\varepsilon$ in~$\L^{\exparen{s_k}}$. 

Note that for any sequence $(s_k)$ and indices $i\le j$, the word~$\alpha^{\exparen{s_k}}_j$ can be expressed as a sequence of concatenations of the words~$\alpha^{\exparen{s_k}}_i$ and~$\alphabar^{\exparen{s_k}}_i$. That is, if~${i\le j}$, then~$\alpha^{\exparen{s_k}}_j$ can be expressed as a word over the alphabet $\{\alpha^{\exparen{s_k}}_i, \alphabar^{\exparen{s_k}}_i\}$, and we use this viewpoint several times in what follows. One consequence is that if~$w$ is a factor of~$\alpha^{\exparen{s_k}}_i$ for some $i$, then it is also a factor of~$\alpha^{\exparen{s_k}}_j$ for all~$j\ge i$. Thus, the language~$\L^{\exparen{s_k}}$ can also be defined as the set of all words that appear as factors of~$\alpha^{\exparen{s_k}}_i$ for all sufficiently large values of $i$, rather than simply one value of $i$.

In the other direction, the next result, which is also essentially from Pouzet's thesis, guarantees that 
all sufficiently long words in~$\L^{\exparen{s_k}}$ actually contain~$\alpha^{\exparen{s_k}}_i$.

\begin{proposition}\label{prop-contains-alpha}
For every sequence $(s_k)$ of positive integers, there exists a function $f^{\exparen{s_k}}(i)$ such that every word in~$w\in \L^{\exparen{s_k}}$ of length at least $f^{\exparen{s_k}}(i)$ contains~$\alpha^{\exparen{s_k}}_i$ as a factor.
\end{proposition}

\begin{proof}
Let~$\alpha=\alpha^{\exparen{s_k}}_{i+1}$. We claim that we can take
\[
	f^{\exparen{s_k}}(i)=2|\alpha|=2^{i+2}s_1s_2\cdots s_{i}.
\]
To see this, let~$w\in\L^{\exparen{s_k}}$ have length at least $f^{\exparen{s_k}}(i)$. Because~$w\in\L^{\exparen{s_k}}$, $w$ must be a factor of~$\alpha^{\exparen{s_k}}_j$ for some $j$, and since
\[
	|w|\ge f^{\exparen{s_k}}(i)>|\alpha^{\exparen{s_k}}_i|,
\]
we see that~$w$ must be a factor of~$\alpha^{\exparen{s_k}}_j$ for some $j\ge i+1$. All such words are composed of concatenations of~$\alpha$ with its complement~$\alphabar$, and thus~$w$ must contain either~$\alpha$ or~$\alphabar$ as a factor, since~$w$ is at least twice as long as both. The result follows because each of~$\alpha$ and~$\alphabar$ contain~$\alpha^{\exparen{s_k}}_i$ as a factor.
\end{proof}

It should be noted that Proposition~\ref{prop-contains-alpha} could be regarded as an instance of "uniform recurrence", if we were to consider a suitable limit object (that is, an infinite word, like the Prouhet--Thue--Morse sequence) for the construction. However, the above statement suffices for our purposes.

That these languages are wqo now follows readily from Proposition~\ref{prop-contains-alpha}.

\begin{proposition}
\label{prop-Lsk-wqo}
For every sequence $(s_k)$ of positive integers, the set~$\L^{\exparen{s_k}}\subseteq\{0,1\}^\ast$ of words is wqo under the factor order.
\end{proposition}

%
%

\begin{proof}
Suppose to the contrary that there were an infinite antichain~$w_1$,~$w_2$, $\dots \in \L^{\exparen{s_k}}$. 
The word~$w_1$ is contained in~$\alpha^{\exparen{s_k}}_i$ for some index $i$ because it lies in~$\L^{\exparen{s_k}}$. Letting $f^{\exparen{s_k}}$ denote the function from Proposition~\ref{prop-contains-alpha}, we see that every word in~$\L^{\exparen{s_k}}$ of length at least $f^{\exparen{s_k}}(i)$ contains~$\alpha^{\exparen{s_k}}_i$ as a factor, and hence also contains~$w_1$ as a factor. It follows that there is some index~$j$ such that~${|w_j|\geq f^{\exparen{s_k}}(i)}$, from which we conclude that~$w_1$ is a factor of~$w_j$, contradicting our assumption that these words form an infinite antichain.
\end{proof}

Pouzet~\cite{pouzet:sur-la-theorie-:} established that $\L^{\exparen{s_k}}$ and~$\L^{\exparen{t_k}}$ are distinct languages whenever $(s_k)$ and $(t_k)$ are distinct sequences. As there are uncountably many sequences of positive integers, there must then be uncountably many wqo factor-closed languages over a binary alphabet. For our enumerative goal, we require something a bit more precise: not only must~$\L^{\exparen{s_k}}$ and~$\L^{\exparen{t_k}}$ be distinct languages, but there must be some length where they contain a different number of words. In fact, the result we prove is stronger still: the set of words of some given length in one language is a \emph{proper subset} of the words of that length in the other.

Before we state and prove this, however, we require a technical characterisation of the possible embeddings of~$\alpha^{\exparen{s_k}}_i$ in~$\alpha^{\exparen{s_k}}_j$ for given $i$ and any $j>i$. We again adopt the viewpoint that~$\alpha_j^{\exparen{s_k}}$ can be regarded as a word over the alphabet $\{\alpha_i^{\exparen{s_k}},\alphabar_i^{\exparen{s_k}}\}$, but to ease exposition let us use~$\alpha_i^\ast$ and~$\alphabar_i^\ast$ to denote the \emph{letters}, and~$\alpha_j^\ast$ the word over $\{\alpha_i^\ast,\alphabar_i^\ast\}$,
with the property that after performing the substitutions~$\alpha_i^\ast\mapsto \alpha_i^{\exparen{s_k}}$ and~$\alphabar_i^\ast\mapsto\alphabar_i^{\exparen{s_k}}$, the word $\alpha_j^\ast$ expands to the binary word~$\alpha_j^{\exparen{s_k}}$.

\begin{proposition}\label{prop-alpha-i-embeddings}%
For every sequence $(s_k)$ of positive integers and integers $i<j$, the only factors of~$\alpha^{\exparen{s_k}}_j$ that are equal to~$\alpha^{\exparen{s_k}}_i$ are given by 
\begin{enumerate}
\item[(i)] the terms corresponding to a letter~$\alpha^{\ast}_i$ of~$\alpha_j^\ast$, and 
\item[(ii)] the middle terms corresponding to a pair of letters~$\alphabar^{\ast}_i\alphabar^{\ast}_i$ in~$\alpha_j^\ast$.
\end{enumerate}
An analogous statement holds for factors of~$\alpha_j^{\exparen{s_k}}$ that are equal to~$\alphabar^{\exparen{s_k}}_i$.
\end{proposition}

\begin{proof}
Since we are dealing with a single sequence $(s_k)$ in this proof, we drop the $(s_k)$ superscript throughout. Hence, the superscripts on words here are all powers.

An occurrence of~$\alpha_i^\ast$ in~$\alpha_j^\ast$ of course corresponds to an~$\alpha_i$ factor of~$\alpha_j$. Similarly, we have
\[
	\alphabar_i\alphabar_i
	=
	\left(\alphabar_{i-1}^{s_{i-1}}\alpha_{i-1}^{s_{i-1}}\right)\left(\alphabar_{i-1}^{s_{i-1}}\alpha_{i-1}^{s_{i-1}}\right)
	=
	\alphabar_{i-1}^{s_{i-1}}\left(\alpha_{i-1}^{s_{i-1}}\alphabar_{i-1}^{s_{i-1}}\right)\alpha_{i-1}^{s_{i-1}}
	=
	\alphabar_{i-1}^{s_{i-1}}\alpha_{i}\alpha_{i-1}^{s_{i-1}},
\]
which demonstrates that the middle terms of~$\alphabar_i\alphabar_i$ also constitute an~$\alpha_i$ factor.

To show that these are the only~$\alpha_i$ factors in~$\alpha_j$, we proceed by induction on $i\ge 1$. In the case $i=1$, we have~$\alpha_i=01$ and~$\alphabar_i=10$. For any $j\ge 2$, consider an occurrence of $01$ in~$\alpha_j$. If this occurrence starts on an odd-indexed letter of~$\alpha_j$, then in~$\alpha_j^\ast$ this occurrence corresponds either to the letter~$\alpha_1^\ast$ or~$\alphabar_1^\ast$, but it clearly cannot be the second of these. If, on the other hand, this occurrence starts on an even-indexed letter of~$\alpha_j$, then this occurrence straddles two letters of~$\alpha_j^\ast$. By inspection, the only possibility is~$\alphabar_i^\ast\alphabar_i^\ast$. A similar argument applies for the occurrences of~$\alphabar_i$ in~$\alpha_j$, and this completes the base case.

Suppose that the proposition is true for some $i\geq 1$. Take $j\ge i+2$, and consider an occurrence of~$\alpha_{i+1}$ in~$\alpha_j$. Since~$\alpha_{i+1} = \alpha_{i}^{s_{i}}\alphabar_{i}^{s_{i}}$, we begin by considering the occurrences of~$\alpha_{i}^{s_{i}}$ and~$\alphabar_{i}^{s_{i}}$ in~$\alpha_j$. 

By induction, the only occurrences of~$\alpha_i$ in~$\alpha_j$ are as given in the proposition. Consequently, the only occurrences of~$\alpha_i^{s_i}$ in~$\alpha_j$ correspond to factors $(\alpha_i^\ast)^{s_i}$ in~$\alpha_j^\ast$, or the middle terms of the binary word corresponding to $(\alphabar_i^\ast)^{s_i+1}$. A similar statement holds for the occurrences of~$\alphabar_i^{s_i}$ in~$\alpha_j$.

We now consider the possible positions for our occurrence of~$\alpha_{i+1} = \alpha_{i}^{s_{i}}\alphabar_{i}^{s_{i}}$ in~$\alpha_j$. If the first half (the word~$\alpha_{i}^{s_{i}}$) appears in the middle of some factor of the form~$\alphabar_i^{s_i+1}$, then the second half,~$\alphabar_i^{s_i}$, cannot embed. Similarly, if the second half of the word embeds in the middle of a factor of the form~$\alpha_i^{s_i+1}$ then the first half cannot embed. Therefore, the only embeddings of~$\alpha_{i+1}$ in~$\alpha_j$ must correspond precisely to instances of the factor $(\alpha_i^\ast)^{s_i}(\alphabar_i^\ast)^{s_i}$ in~$\alpha_j^\ast$. A similar statement applies to~$\alphabar_{i+1}$. 

Finally, by construction,~$\alpha_j^\ast$ is composed of factors of the form $(\alpha_i^\ast)^{s_i}(\alphabar_i^\ast)^{s_i}$ and $(\alphabar_i^\ast)^{s_i}(\alpha_i^\ast)^{s_i}$. These correspond to the words~$\alpha_{i+1}$ and~$\alphabar_{i+1}$ that can be used to make up~$\alpha_j$, from which the inductive step follows. 
\end{proof}

We now state the main result of this section; we prove it after considering an example.

\begin{proposition}
\label{prop-words-distinct-enum}
Suppose that $(s_k)$ and $(t_k)$ are distinct sequences of positive integers, and that~$(s_k)$ lexicographically precedes $(t_k)$. Then there exists an integer $M\ge 4$ such that~$\L^{\exparen{s_k}}_n = \L^{\exparen{t_k}}_n$ for all~${n<M}$, but~$\L^{\exparen{s_k}}_M\supsetneq \L^{\exparen{t_k}}_M$.
\end{proposition}

We note that the ``inequalities'' in Proposition~\ref{prop-words-distinct-enum} are in some sense reversed---if $(s_k)$ precedes $(t_k)$ in lexicographical order, then it follows that~$|\L^{\exparen{s_k}}_M|$ is \emph{greater} than~$|\L^{\exparen{t_k}}_M|$. We illustrate this with the concrete example where the sequences are $(1,1,1,\dots)$ and $(2,1,1,\dots)$. In this case the construction gives us the table below.
\[
	\begin{array}{rll}\\
	i&\alpha^{(1,1,1,\dots)}_i&\alpha^{(2,1,1,\dots)}_i\\[2pt]
	\hline\\[-10pt]
	1&01					&01	\\[2pt]
	2&01\,10				&01\,01\,10\,10\\[2pt]
	3&0110\,1001			&01011010\,10100101\\[2pt]
	4&01101001\,10010110	&0101101010100101\,1010010101011010
	\end{array}
\]
One can check that the languages~$\L^{(1,1,1,\dots)}$ and~$\L^{(2,1,1,\dots)}$ agree up to and including length $3$:
\[
\begin{array}{lclcl}
	\L^{(2,1,1,\dots)}_0&=&\L^{(1,1,1,\dots)}_0	&=&	\{\varepsilon\},\\[2pt]
	\L^{(2,1,1,\dots)}_1&=&\L^{(1,1,1,\dots)}_1	&=&	\{0, 1\},\\[2pt]
	\L^{(2,1,1,\dots)}_2&=&\L^{(1,1,1,\dots)}_2	&=&	\{00, 01, 10, 11\},\\[2pt]
	\L^{(2,1,1,\dots)}_3&=&\L^{(1,1,1,\dots)}_3	&=&	\{001, 010, 011, 100, 101, 110\};
\end{array}
\]
but they differ at length $4$:
\[
	\begin{array}{lcl}
		\L^{(2,1,1,\dots)}_4		&=&	\{0010, 0100, 0101, 0110, 1001, 1010, 1011, 1101\},\\[2pt]
		\L^{(1,1,1,\dots)}_4		&=&	\{0010, 0100, 0101, 0110, 1001, 1010, 1011, 1101, 0011, 1100\}.
	\end{array}
\]
Hence, in this case Proposition~\ref{prop-words-distinct-enum} holds with $M=4$. We conclude this section by proving Proposition~\ref{prop-words-distinct-enum}, and then we begin translating Pouzet's construction to the permutation context in the next section.

\begin{proof}[Proof of Proposition~\ref{prop-words-distinct-enum}]
Let $(s_k)$ and $(t_k)$ be as in the statement of the result, and choose $I$ so that~${s_i=t_i}$ for all indices $i<I$, but $s_I<t_I$. We claim that we may take ${M=2^{I}s_1s_2\cdots s_I+2}$. Note that~$M\ge 4$.
For convenience, set~$\alpha=\alpha^{\exparen{s_k}}_{I}=\alpha^{\exparen{t_k}}_{I}$, and note that $|\alpha|=2^{I}s_1\cdots s_{I-1}$ and~${M=s_I|\alpha|+2}$.

First consider a word~$w$ of length $i<M$ in~$\L^{\exparen{s_k}}$, and let $j>I$ be such that~$\alpha_j^{\exparen{s_k}}$ contains~$w$. Now view~$\alpha_j^{\exparen{s_k}}$ as a sequence of concatenations of~$\alpha^{s_I}$ and~$\alphabar^{s_I}$. Since~$w$ has length $i\le s_I|\alpha|+1$, it follows that every~$w$ embedding in~$\alpha_j^{\exparen{s_k}}$ is contained in at most two terms of this concatenation. That is,~$w$ appears as a factor in one of the following words:
\[
	\alpha^{s_I}\alpha^{s_I},\quad
	\alpha^{s_I}\alphabar^{s_I},\quad
	\alphabar^{s_I}\alpha^{s_I},\quad
	\alphabar^{s_I}\alphabar^{s_I}
\]
In fact, containment in one of these four words is a precise characterization of the words of length at most $M-1$ in~$\L^{\exparen{s_k}}$, since all four words appear in~$\alpha_{I+3}^{\exparen{s_k}}$:
\[
	\alpha_{I+3}^{\exparen{s_k}} = \big(
		\left(\alpha^{s_I}\alphabar^{s_I}\right)^{s_{I+1}}
		\left(\alphabar^{s_I}\alpha^{s_I}\right)^{s_{I+1}}\big)^{s_{I+2}}
	\big(
		\left(\alphabar^{s_I}\alpha^{s_I}\right)^{s_{I+1}}
		\left(\alpha^{s_I}\alphabar^{s_I}\right)^{s_{I+1}}\big)^{s_{I+2}}
\]
Furthermore, since $t_I > s_I$, all four of these words also appear in~$\L^{\exparen{t_k}}$, and this establishes that~$\L^{\exparen{s_k}}$ and~$\L^{\exparen{t_k}}$ contain the same words up to length $M-1$.

We now turn our attention to the words of length $M$. First, any word~$w$ of length $M$ that embeds in one of the four words
\[
	\alpha^{s_I}\alpha^{s_I},\quad
	\alpha^{s_I}\alphabar^{s_I},\quad
	\alphabar^{s_I}\alpha^{s_I},\quad
	\alphabar^{s_I}\alphabar^{s_I}
\]
will lie in both~$\L^{\exparen{s_k}}$ and~$\L^{\exparen{t_k}}$, by the same argument we have just given. 
Any other word of length $M$ must be formed from a copy of~$\alpha^{s_I}$ or~$\alphabar^{s_I}$, with exactly one letter before and one after. That is, the only remaining words to consider are the following:
\begin{gather*} 
0\alpha^{s_I}0,\quad
1\alpha^{s_I}0,\quad
0\alpha^{s_I}1,\quad
1\alpha^{s_I}1,\\
0\alphabar^{s_I}0,\quad
1\alphabar^{s_I}0,\quad
0\alphabar^{s_I}1,\quad
1\alphabar^{s_I}1.	
\end{gather*}
Let us assume that $I$ is even; the case where $I$ is odd is analogous. By Observation~\ref{obs-last-letters}, the word~$\alpha$ both begins and ends with 0. The following table summarizes whether each of the eight words above belongs to~$\L^{\exparen{s_k}}$ and~$\L^{\exparen{t_k}}$, and if so illustrates how it arises (note that the words specified in the second and third columns are factors of~$\alpha_{I+3}^{\exparen{s_k}}$ and~$\alpha_{I+3}^{\exparen{t_k}}$, respectively.

\begin{center}
\begin{tabular}{c@{\qquad}l@{\qquad}l}
word 				&in~$\L^{\exparen{s_k}}$, factor of	&in~$\L^{\exparen{t_k}}$, factor of\\
\hline\\[-10pt]
$0\alpha^{s_I}0$	&$\alphabar^{s_I}\alphabar^{s_I}$ (see below)				&$\alpha^{t_I}\alpha^{t_I}$\\
$1\alpha^{s_I}0$	&$\alphabar^{s_I}\alpha^{s_I}\alpha^{s_I}$	&$\alphabar^{t_I}\alpha^{t_I}$\\
$0\alpha^{s_I}1$	&$\alpha^{s_I}\alpha^{s_I}\alphabar^{s_I}$	&$\alpha^{t_I}\alphabar^{t_I}$\\
$1\alpha^{s_I}1$	&$\alphabar^{s_I}\alpha^{s_I}\alphabar^{s_I}$&not in set (see below)\\
$0\alphabar^{s_I}0$	&$\alpha^{s_I}\alphabar^{s_I}\alpha^{s_I}$	&not in set (see below)\\
$1\alphabar^{s_I}0$	&$\alphabar^{s_I}\alphabar^{s_I}\alpha^{s_I}$	&$\alphabar^{t_I}\alpha^{t_I}$\\
$0\alphabar^{s_I}1$	&$\alpha^{s_I}\alphabar^{s_I}\alphabar^{s_I}$	&$\alpha^{t_I}\alphabar^{t_I}$\\
$1\alphabar^{s_I}1$	&$\alpha^{s_I}\alpha^{s_I}$ (see below)		&$\alphabar^{t_I}\alphabar^{t_I}$
\end{tabular}
\end{center}
There remain four entries in the above table to consider. Let us begin first with the word $0\alpha^{s_I}0$. Proposition~\ref{prop-alpha-i-embeddings} tells us that~$\alpha$ embeds in the middle of~$\alphabar\alphabar$, and thus~$\alpha^{s_I}$ embeds in the middle of~$\alphabar^{s_I}\alphabar^{s_I}$. Furthermore, the letter immediately to the left of this embedding is the last letter of~$\alphabar_{I-1}^{\exparen{s_k}}$, which is 0 by Observation~\ref{obs-last-letters} (since we are assuming that $I-1$ is odd). Similarly, the first letter after this embedding is the first letter of~$\alphabar_{I-1}^{\exparen{s_k}}$, which is also 0, and hence $0\alpha^{s_I}0$ is a factor of~$\alphabar^{s_I}\alphabar^{s_I}$.

A similar argument can be applied to show that $1\alphabar^{s_I}1\in\L^{\exparen{s_k}}$, and this establishes the containment~${\L_M^{\exparen{s_k}}\supseteq \L_M^{\exparen{t_k}}}$.

Our final task is to show that neither $1\alpha^{s_I}1$ nor $0\alphabar^{s_I}0$ lies in~$\L^{\exparen{t_k}}$. We consider only $1\alpha^{s_I}1$, the case for the other word being entirely analogous. 

By Proposition~\ref{prop-alpha-i-embeddings}, the only factors of~$\alpha_j^{\exparen{t_k}}$ that are equal to~$\alpha$ are either given by the \emph{letter}~$\alpha$, or appear in the middle of the pair of letters~$\alphabar\alphabar$, when we express~$\alpha_j^{\exparen{t_k}}$ as a word over $\{\alpha,\alphabar\}$. Thus, the binary word~$\alpha^{s_I}$ appears as a factor of~$\alpha_j^{\exparen{t_k}}$ only as the sequence of letters~$\alpha^{s_I}$, or in the middle of~$\alphabar^{s_I+1}$.
In this latter embedding, the letter in~$\alphabar^{s_I+1}$ that lies immediately to the left of such an embedding is the last letter of~$\alphabar_{I-1}^{\exparen{s_k}}$, which we have already established is equal to 0. Thus $1\alpha^{s_I}1$ does not appear in the middle of~$\alphabar^{s_I+1}$.

The only remaining possibility is that $1\alpha^{s_I}1$ embeds into~$\alpha_j^{\exparen{t_k}}$ precisely as~$\alpha^{s_I}$, together with one extra letter on either side. Now~$\alpha_j^{\exparen{t_k}}$ can be written as a word comprising factors of the form~$\alpha^{t_I}$ and~$\alphabar^{t_I}$. Since $t_I > s_I$, we are forced to embed either the letter 1 on the left of $1\alpha^{s_I}1$ as the rightmost letter of~$\alpha$, or the letter 1 on the right of $1\alpha^{s_I}1$ as the leftmost letter of~$\alpha$. Since~$\alpha$ both begins and ends with 0, however, neither case is possible. Thus $1\alpha^{s_I}1$ does not embed into~$\alpha_j^{\exparen{t_k}}$, from which we conclude that $1\alpha^{s_I}1\not\in\L^{\exparen{t_k}}$, as required.
\end{proof}

\section{Rightward-Yearning Pin Sequences}
\label{sec-rightward-yearning}

Our tool for translating the results about Pouzet's languages~$\L^{\exparen{s_k}}$ to the permutation context are pin sequences, first introduced by Brignall, Huczynska, and Vatter~\cite{brignall:decomposing-sim:}, although the language-theoretical approach we take is closer to that used by Brignall, Ru\v{s}kuc, and Vatter~\cite{brignall:simple-permutat:decide:}. Pin sequences are best described with the pictorial description of the permutation pattern order already depicted in Figure~\ref{fig-pattern-containment}. The \emph{plot} of the permutation~$\pi$ is the set $\{(i,\pi(i))\}$ of points. Clearly every plot of a permutation is \emph{generic} in the sense that no two of its points share the same $x$- or $y$-coordinate. Conversely, every finite generic set of points in the plane is order isomorphic to the plot of a unique permutation, in the sense that two sets of points in the plane are \emph{order isomorphic} if the axes can be stretched and shrunk to transform one of the sets into the other.

An \emph{axis-parallel rectangle} is any rectangle in the plane with sides parallel to the~$x$- and~$y$-axes. The \emph{rectangular hull} of a set of points in the plane is defined as the smallest axis-parallel rectangle containing them. Given a sequence $(p_1, \dots, p_i)$ of points in the plane, a \emph{proper pin} for this sequence is a point $p$ that lies outside their rectangular hull and \emph{separates}~$p_i$ from $\{p_1,\dots,p_{i-1}\}$, meaning that $p$ lies either horizontally or vertically between~$p_i$ and the rectangular hull of $\{p_1,\dots,p_{i-1}\}$. A \emph{proper pin sequence} is then constructed by starting with two points~$p_1$ and~$p_2$ (whose placement we discuss later), choosing~$p_3$ to be a proper pin for $(p_1,p_2)$, then choosing~$p_4$ to be a proper pin for $(p_1,p_2,p_3)$, and so on. We describe pins as either \emph{left}, \emph{right}, \emph{up}, or \emph{down} based on their position relative to the rectangular hull of $\{p_1,\dots,p_{i}\}$. Note that the direction of a pin uniquely specifies its position relative to the previous points in a pin sequence. We specify pin sequences with the alphabet $\{\mathsf{l},\r,\u,\d\}$.

It follows from their definition that proper pin sequences must turn by $90^\circ$ with each pin. In other words, an up pin may be immediately followed by a left or a right pin, but not by another up pin or by a down pin. As our goal is to encode binary strings as proper pin sequences, we \emph{could} translate each~$0$ into left or down and each~$1$ into right or up, with the choices determined by the previous pin. If we were to do this, we might for example use the correspondence
\[
	01101001
	\mapsto
	\mathsf{drululdr}.
\]

While we suspect that our results would remain true with this translation from binary words to proper pin sequences, they would undoubtedly be more troublesome to prove. Instead, we restrict our attention to \emph{rightward-yearning pin sequences}, defined as pin sequences that alternate between up or down pins at odd indices and right pins at even indices. Such restricted pin sequences are the subject of Jarvis' thesis~\cite{jarvis:pin-classes:}, and one of his results gives the basis and enumeration for a symmetry of the rightward-yearning pin sequences.

\begin{theorem}[Jarvis~{\cite[Theorem 7.24, after a symmetry]{jarvis:pin-classes:}}]
\label{thm-rightward-basis-enum}
The downward closure of the set of all rightward-yearning pin sequences is the class
\[
	\Av(1432, 4123, 13452, 14523, 24513, 42153, 52143, 53214).
\]
The generating function of this class is
\[
	\frac{1-x-4x^2-4x^3-8x^4-4x^5}{1-2x-4x^2-2x^3-8x^4-4x^5}.
\]
\end{theorem}

There is a natural correspondence between binary words and rightward-yearning pin sequences in which~$0$ maps to a down pin followed by a right pin, and~$1$~maps to an up pin followed by a right pin, so
\[
	01101001
	\mapsto
	\mathsf{drururdrurdrdrur}.
\]
For technical reasons, we refine this encoding by subscripting each right pin with the type of pin immediately preceding it, thereby slightly modifying our alphabet to $\alphabet$. The example above then becomes
\[
	01101001
	\mapsto
	\mathsf{d r_d u r_u u r_u d r_d u r_u d r_d d r_d u r_u}.
\]
(This pin sequence is plotted in Figure~\ref{fig-rightward-yearning}.) We stress that the subscripts on our encodings of right pins do not affect the actual pin sequences; these subscripts instead inform us where the corresponding points lie in the plane.

Formally speaking, given a binary word~$w\in\{0,1\}^\ast$, we denote by $\rho(w)$ the word (of twice the length as~$w$) over the alphabet $\alphabet$ that is obtained by performing the substitutions
\[
	0\leftarrow \d\rd
	\quad
	\text{and}
	\quad
	1\leftarrow \u\ru;
\]
that is, replacing occurrences of~$0$ by $\d\rd$ and occurrences of~$1$ by $\u\ru$.
We are frequently interested only in the image of $\rho$ and factors of those words, and so we define the language
\[
	\P			=
		\text{the factor closure of $\rho(\{0,1\}^\ast)$}.
\]

\begin{figure}
\begin{center}
\begin{tikzpicture}[scale=0.3]
\node[openpoint] (orig) at (0,-0.5) {};
\foreach \y [count=\x] in {-2,1,-1,3,0,-4,2,5,-3,-6,4,-8,-5,7,-7,6}
	\node[point] (\y) at (\x,\y) {};
\foreach \node [count=\n,remember=\node as \prevnode (initially orig)] in {-2,-1,1,0,3,2,-4,-3,5,4,-6,-5,-8,-7,7,6} {
	\ifodd\n
		\draw (\prevnode-|\node) -- ++(0,0.5)-- ++(0,-1) -- (\node);
	\else
		\draw (\prevnode|-\node) -- ++(-0.5,0) -- (\node);
	\fi
	}
\end{tikzpicture}
\end{center}
\caption{The rightward-yearning pin sequence associated to the word $\mathsf{d r_d u r_u u r_u d r_d u r_u d r_d d r_d u r_u}$.}
\label{fig-rightward-yearning}
\end{figure}

It remains to explain how to start a pin sequence. In the present work, we start every pin sequence with a point called an \emph{origin} and labelled by~$p_0$. \emph{Importantly, we do not consider the origin~$p_0$ to be part of the resulting rightward-yearning pin sequence.} 

For the first real pin,~$p_1$, if it has encoding $\u$ or $\ru$, then~$p_1$ appears above and to the right of~$p_0$. Analogously, if~$p_1$ has encoding $\d$ or $\rd$, then it is placed below and to the right of~$p_0$. In either case, the second pin~$p_2$ slices the rectangular hull of $(p_0,p_1)$ in the direction indicated by its encoding.

As our pin sequences progress only to the right, this origin~$p_0$ lies to the left of all other pins of the pin sequence, and will lie below all pins whose encoding is $\mathsf{u}$ or $\mathsf{r_u}$, and above all pins whose encoding is $\mathsf{d}$ or $\mathsf{r_d}$. In this way, the origin partitions each of the entries of the pin sequence into two halves, and our pin sequences could be considered a simple type of \emph{grid pin sequence} (as first considered by Brignall~\cite{brignall:grid-classes-an:}), but we do not adopt this viewpoint.

Given any word~$w\in\P$ of length~$n$, we take the \emph{rightward-yearning pin sequence defined by~$w$} to be the permutation $\psi_w^\norigin$ that is order isomorphic to the set $\{p_1,p_2,\dots,p_n\}$ of points defined by the word~$w$. Thus $|\psi_w^\norigin|=|w|=n$.

As is demonstrated in Section~\ref{sec-decomp}, in order to discuss subpermutations of rightward-yearning pin sequences, it is necessary to also introduce the permutations $\psi_w^\origin$ that include the origin as an extra point. We take $\psi^\origin_w$ to be the permutation order isomorphic to the set $\{p_0,p_1,p_2,\dots,p_n\}$ of points defined by the word~$w$, so $|\psi_w^\origin|=|w|+1=n+1$.

\section{The Construction}\label{sec-construction}

We now have the necessary background to describe the family of permutation classes used to prove Theorem~\ref{thm-wqo-not-algebraic}. Our construction begins with the binary words of Section~\ref{sec-pouzet-factors}, for which we recall that
\begin{align*}
	\L^{\exparen{s_k}}	&=
		\text{the factor-closed wqo languages of Section~\ref{sec-pouzet-factors}}.\\
\intertext{We then consider the factors of the image under $\rho$ of those languages, and the class of permutations to which they correspond:}
	\P^{\exparen{s_k}}	&=
		\text{the factor closure of $\rho(\L^{\exparen{s_k}})$ and}\\
	\C^{\exparen{s_k}}	&=
		\text{the downward closure of $\{\psi_w^\norigin : w\in\P^{\exparen{s_k}}\}$}.
\end{align*}

It follows from these definitions that $\P$,~$\L^{\exparen{s_k}}$, and $\P^{\exparen{s_k}}$ are closed under taking factors, while~$\C^{\exparen{s_k}}$ is closed under taking subpermutations, and is thus a permutation class. What remains to do is to establish that the classes $\C^{\exparen{s_k}}$ are wqo, and that if the sequences $(s_k)$ and $(t_k)$ differ, then the resulting classes $\C^{\exparen{s_k}}$ and $\C^{\exparen{t_k}}$ have different enumerations. These two results will follow by lifting the analogous results about the languages~$\L^{\exparen{s_k}}$, Propositions~\ref{prop-Lsk-wqo} and \ref{prop-words-distinct-enum}, to this context. In order to do this, we must first establish a decomposition result for the members of $\C^{\exparen{s_k}}$ in the next section. Before that, we make a simple observation now that we have the terminology to express it.

\begin{proposition}
\label{prop-Psk-to-Csk}
Suppose that~$v$ is a factor of~$w$ for words $v,w\in\P$.
Then, the permutation~$\psi_v^\norigin$ is contained in the permutation~$\psi_w^\norigin$ and the permutation~$\psi_v^\origin$ is contained in the permutation~$\psi_w^\origin$.
\end{proposition}
\begin{proof}
Fix a factor of~$v$ in~$w$. The pins of $\psi_w^\norigin$ corresponding to this factor are in the same relative position to each other as the pins of~$\psi_v^\norigin$, which verifies the claim for this pair of permutations. For the version of the result with origins, we note that the origin in~$\psi_w^\origin$ is in the same position relative to the pins that correspond to a factor of~$v$ in~$w$ as the origin in~$\psi_v^\origin$ is in relative to the rest of the pins of that permutation.
\end{proof}

\section{A Decomposition}
\label{sec-decomp}

We frame the discussion in this section as considering the effect of deleting points from rightward-yearning pin sequences. There are essentially three types of pins we can delete: the first pin, the last pin, or an interior pin. We handle these cases below in order of their difficulty, for a word~$w\in\P$ of length $n$.
\begin{itemize}
\item Deleting the last pin from a sequence is equivalent to not creating it at all. Thus the permutation obtained from~$\psi^\norigin_w$ by deleting the point corresponding to~$p_n$ is~$\psi^\norigin_{w(1)\cdots w(n-1)}$.
\item The same argument as above holds for the first pin, so the permutation obtained from~$\psi^\norigin_w$ by deleting the point corresponding to the first pin is~$\psi^\norigin_{w(2)\cdots w(n)}$. 
\item Deleting the $i$th pin from~$\psi^\norigin_w$ for some index $2\le i\le n-2$ corresponds to replacing the origin in the permutation~$\psi^\origin_{w(i+1)\cdots w(n)}$ with the permutation~$\psi^\norigin_{w(1)\cdots w(i-1)}$.
\end{itemize}
We call the operation in this last case \emph{inflating the origin}, and this case is the reason we introduced the permutations~$\psi^\origin_w$ in Section~\ref{sec-rightward-yearning} in the first place. The process of removing pins from the permutations~$\psi^\origin_w$ is also needed, but is entirely analogous to the above, and will be handled once we have introduced some additional notation.

Inflating origins is similar to the sum of two permutations. Recall that given a permutation~$\sigma$ of length $m$ and another permutation~$\psi$ of length $n$, their \emph{sum} is the permutation~$\sigma\oplus\tau$ of length~$m+n$ defined by
\[
	(\sigma\oplus\tau)(i)
	=
	\left\{\begin{array}{ll}
	\sigma(i)&\text{if $1\le i\le m$,}\\
	\tau(i-m)+m&\text{if $m+1\le i\le m+n$.}
	\end{array}\right.
\]

The origin in~$\psi^\origin_w$ is always the leftmost point, so we define the more general operation of inflating the first entry of a permutation. Suppose that~$\sigma$ is a permutation of length $m$ and that $\tau$ is a permutation of length $n+1$. Then, we define~$\sigma\boxplus\tau$ as the permutation of length $m+n$ obtained by inflating the first entry of $\tau$ by~$\sigma$. Formulaically, $\sigma\boxplus\tau$ can be defined by
\[
	(\sigma\boxplus\tau)(i)
	=
	\left\{\begin{array}{ll}
	\sigma(i)+\tau(1)-1&\text{if $1\le i\le m$,}\\
	\tau(i-m+1)&\text{if $m< i\le m+n$ and $\tau(i-m+1)<\tau(1)$, and}\\
	\tau(i-m+1)+m-1&\text{if $m< i\le m+n$ and $\tau(i-m+1)>\tau(1)$.}
	\end{array}\right.
\]
See Figure~\ref{fig-interior-pin} for an example. There is some relationship between $\oplus$ and $\boxplus$ in that, if $\sigma$ has length at least two, then $\sigma\oplus\tau=\sigma\boxplus(1\oplus\tau)$.

\begin{figure}
{\centering
\begin{tikzpicture}[scale=0.25]
\node[openpoint] (orig) at (0,0) {};
\foreach \y [count=\x] in {-2,2,-1,-4,1,4,-3,-6,3,-5}
	\node[point] (\y) at (\x,\y) {};
\foreach \node [count=\n,remember=\node as \prevnode (initially orig)] in {-2,-1,2,1,-4,-3,4,3,-6,-5} {
	\ifodd\n
		\draw (\prevnode-|\node) -- ++(0,0.5)-- ++(0,-1) -- (\node);
	\else
		\draw (\prevnode|-\node) -- ++(-0.5,0) -- (\node);
	\fi
	}
\node[openpoint,minimum width=\plotptradius+3pt] at (-4) {};
\node[draw=none,fill=none] at (12,-1) {$\rightarrow$};
\begin{scope}[shift={(15,-0.5)}]
\draw[fill=black!10,draw=none] (-0.5,-2.5) rectangle (4.5,2.5);
\node[openpoint] (orig) at (0,0) {};
\coordinate (-40) at (4,-4);
\foreach \y [count=\x] in {-2,2,-1,1}
	\node[point] (\y) at (\x,\y) {};
\foreach \y [count=\x] in {4,-3,-5,3,-4}
	\node[point] (\y) at (\x+4,\y) {};
\foreach \node [count=\n,remember=\node as \prevnode (initially orig)] in {-2,-1,2,1,4,3} {
	\ifodd\n
		\draw (\prevnode-|\node) -- ++(0,0.5)-- ++(0,-1) -- (\node);
	\else
		\draw (\prevnode|-\node) -- ++(-0.5,0) -- (\node);
	\fi
	}
\foreach \node [count=\n,remember=\node as \prevnode (initially -40)] in {-3,4,3,-5,-4} {
	\ifodd\n
		\draw (\prevnode|-\node) -- ++(-0.5,0) -- (\node);
	\else
		\draw (\prevnode-|\node) -- ++(0,0.5)-- ++(0,-1) -- (\node);
	\fi
	}
\end{scope}
\end{tikzpicture}\par}
\caption{Deleting the circled interior pin from~$\psi^\norigin_{\d\rd\u\ru\d\rd\u\ru\d\rd}$ results in~$\psi^\norigin_{\d\rd\u\ru}\boxplus\psi^\origin_{\rd\u\ru\d\rd}$.}\label{fig-interior-pin}
\end{figure}

From our previous discussion, it follows that if we delete the $i$th pin from~$\psi^\norigin_w$, where~$w\in\P$ has length~$n$ and $2\le i\le n-1$, then we obtain the permutation
\[
	\psi^\norigin_{w(1)\cdots w(i-1)}\boxplus \psi^\origin_{w(i+1)\cdots w(n)}.
\]
Indeed, letting $\varepsilon$ denote the empty word or empty permutation (as dictated by the context), and with the understanding that~$\psi^\norigin_\varepsilon=\varepsilon$ while~$\psi^\origin_\varepsilon=1$, we see that for \emph{any} $1\le i\le n$, the result of deleting the point corresponding to the~$i$th pin~$p_i$ from~$\psi^\norigin_w$ is
\[
	\psi^\norigin_{w(1)\cdots w(i-1)}\boxplus\psi^\origin_{w(i+1)\cdots w(n)}.
\]

As noted earlier, we must also describe how to delete points from the permutations~$\psi^\origin_w$, which follows by the same analysis. If we delete the origin~$p_0$ from~$\psi^\origin_w$, we obviously obtain the permutation~$\psi^\norigin_w$.
Otherwise, if~$w\in\P$ has length $n$ and $1\le i\le n$, then the result of deleting the point corresponding to the $i$th pin~$p_i$ from~$\psi^\origin_w$ is
\[
	\psi^\origin_{w(1)\cdots w(i-1)}\boxplus\psi^\origin_{w(i+1)\cdots w(n)}.
\]

Our next result puts this decomposition in the form we need it later. A stronger result that puts conditions on how the factors $w_1,\dots,w_k$ must appear in $w$ would be possible, but is not necessary for our uses. Note that we needn't include parentheses in the statement of this result because the operation $\boxplus$ is associative.

\begin{proposition}\label{prop-boxplus-decomp}
For any nonempty subpermutation~$\pi$ of the rightward-yearning pin sequence~$\psi^\norigin_w$, there exist nonempty words~$w_1$, $\dots$,~$w_k\in\P$, each appearing as a factor in~$w$, such that
\[
	\pi=
	\psi^\norigin_{w_1}\boxplus\psi^\origin_{w_2}\boxplus\cdots\boxplus\psi^\origin_{w_k}.
\]
\end{proposition}

\begin{proof}
Fix a particular embedding of~$\pi$ in~$\psi_w^\norigin$, and delete the entries of~$\psi_w^\norigin$ that are not involved in this embedding one at a time. Before we delete any given point, we may assume by induction that we have a permutation of the form~$\psi^\norigin_{u_1}\boxplus\psi^\origin_{u_2}\boxplus\cdots\boxplus\psi^\origin_{u_\ell}$ for some choice of words $u_1$, $\dots$, $u_\ell\in \P$ that appear as factors in the word~$w$. 

Let $j$ denote the index of the word $u_j$ corresponding to the component from which our next entry is to be deleted.
This entry is either any pin of~$\psi^\norigin_{u_1}$ (if $j=1$) or a non-origin pin of~$\psi^\origin_{u_j}$ (if $j>1$). If $|u_j| >1$, then by the considerations before the statement of the proposition it follows that any such deletion results in another permutation of the desired form. If $|u_j|=1$ and $j>1$, then the component $\psi^\origin_{u_j}$ becomes $\psi^\origin_\varepsilon$, and since $\alpha\boxplus \psi^\origin_\varepsilon = \alpha$ we obtain the expression	
\[
\psi^\norigin_{u_1}\boxplus\cdots \boxplus \psi^\origin_{u_{j-1}}\boxplus\psi^\origin_{u_{j+1}}\boxplus\cdots\boxplus\psi^\origin_{u_\ell}.
\]	
Similarly, if $|u_j|=1$ and $j=1$, then from the fact that $\psi^\norigin_\varepsilon \boxplus \beta^\origin = \beta^\norigin$, we get 
$\psi^\norigin_{u_2}\boxplus\cdots\boxplus\psi^\origin_{u_\ell}$. This completes the inductive step and hence the proof of the proposition.
\end{proof}

We are now able to describe the decomposition of all permutations in some class $\C^{\exparen{s_k}}$ in terms of permutations defined by words from the language $\P^{\exparen{s_k}}$.

\begin{theorem}
\label{thm-Csk-description}
Let $(s_k)$ be a sequence of positive integers. The class $\C^{\exparen{s_k}}$ consists precisely of those permutations~$\pi$ that can be expressed as
\[
	\pi=\psi^\norigin_{v_1}\boxplus\psi^\origin_{v_2}\boxplus\cdots\boxplus\psi^\origin_{v_\ell}
\]
for words $v_1$, $v_2$, $\dots$, $v_\ell\in\P^{\exparen{s_k}}$ with $|v_1|+|v_2|+\cdots+|v_\ell|=|\pi|$.
\end{theorem}

\begin{proof}
Let~$\pi\in\C^{\exparen{s_k}}_n$. By the definition of $\C^{\exparen{s_k}}$, there is some word~$w\in\P^{\exparen{s_k}}$ such that~$\pi$ is contained in the permutation~$\psi^\norigin_w$. By Proposition~\ref{prop-boxplus-decomp}, it follows that there exist words $v_1$, $v_2$, $\dots$, $v_\ell$, each appearing as a factor of~$w$, such that 
\[
	\pi=\psi^\norigin_{v_1}\boxplus\psi^\origin_{v_2}\boxplus\cdots\boxplus\psi^\origin_{v_\ell}.
\]
Since each $v_i$ is a factor of~$w\in\P^{\exparen{s_k}}$, it follows that each $v_i\in\P^{\exparen{s_k}}$. It is similarly clear that $|v_1|+|v_2|+\cdots+|v_\ell|=|\pi|$.

Conversely, we consider a permutation~$\pi$ as in the statement of the theorem and appeal to the properties of the language~$\L^{\exparen{s_k}}$. By definition, for each $v_i\in\P^{\exparen{s_k}}$, there exists an integer $j_i$ such that $v_i$ is a factor of $\rho(\alpha_{j_i}^{\exparen{s_k}})$. Set $m=\max\{j_1,\dots,j_\ell\}$, so that each of $v_1$, $\dots$, $v_\ell$ appears as a factor in $\rho(\alpha_m^{\exparen{s_k}})$.

Letting $f^{\exparen{s_k}}$ denote the function from Proposition~\ref{prop-contains-alpha}, we see from that result that every (binary) word of length at least $f^{\exparen{s_k}}(m)$ in~$\L^{\exparen{s_k}}$ contains $\alpha_m^{\exparen{s_k}}$. It follows that every word in $\P^{\exparen{s_k}}$ of length at least $2f^{\exparen{s_k}}(m)+2$ contains $\rho(\alpha_m^{\exparen{s_k}})$, and thus also contains all of $v_1$, $v_2$, $\dots$, $v_\ell$. (The ``$+2$'' here is due to the fact that the first and last letter of a word in $\P$ need not be part of a factor of the form $\rho(u)$.) It follows that every word in $\P^{\exparen{s_k}}$ of length at least $\ell\cdot(2f^{\exparen{s_k}}(m)+2)+(\ell-1)$ must contain a word~$w$ of the form
\[
	w = v_1x_1v_2x_2\cdots x_{\ell-1}v_\ell
\]
where $x_1$, $\dots$, $x_{\ell-1}$ are arbitrary non-empty words. Thus~$\pi$ is a subpermutation of~$\psi^\norigin_w$.
\end{proof}

It is tempting to conclude from Theorem~\ref{thm-Csk-description} that each $\C^{\exparen{s_k}}$ is $\boxplus$-closed, which would mean that ${\sigma,\tau\in\C^{\exparen{s_k}}}$ implies that~$\sigma\boxplus\tau\in\C^{\exparen{s_k}}$, but we need to be careful. If we simply inflate the first entry of $\tau$ by~$\sigma$, then what results is not guaranteed to lie in $\C^{\exparen{s_k}}$, since it is not necessarily the case that~$\tau$ can be described in such a way that its first entry can act as a (non-phantom) origin. However, since~${\tau\in\C^{\exparen{s_k}}}$, Theorem~\ref{thm-Csk-description} tells us that we can write
\[
	\tau=\psi^\norigin_{v_1}\boxplus\psi^\origin_{v_2}\boxplus\cdots\boxplus\psi^\origin_{v_\ell}.
\]
for some $v_1$, $v_2$, $\dots$, $v_\ell$ in $\P^{\exparen{s_k}}$, and then Theorem~\ref{thm-Csk-description} tells us that
\[
	\sigma\boxplus \psi^\origin_{v_1}\boxplus\psi^\origin_{v_2}\boxplus\cdots\boxplus\psi^\origin_{v_\ell}
		\in\C^{\exparen{s_k}}.
\]

\section{Indecomposable Permutations}\label{sec-indecomposable}

One calls a permutation \emph{sum decomposable} if it can be expressed as the sum of two shorter permutations, and \emph{sum indecomposable} otherwise. Analogously, we say that a permutation is \emph{\mbox{$\boxplus$-decomposable}} if it can be expressed as~$\sigma\boxplus\tau$ for two shorter permutations each of length at least two, and \emph{$\boxplus$-indecomposable} otherwise. (We must require that both~$\sigma$ and $\tau$ have length at least two to avoid trivial decompositions, because $1\boxplus\pi=\pi\boxplus 1=\pi$.)

\begin{table}[t]
\vspace{-12pt}
\vtop{\[
	\begin{array}{@{}c@{\quad}c@{}}
		\begin{array}[t]{rcll}
		\hline\\[-10pt]
		1		&=&\psi_\u^\norigin=\psi_\d^\norigin=\psi_\r^\norigin		&\text{$\boxplus$-indecomp.}\\[2pt]
		\hline\\[-10pt]
		12		&=&\psi_{\r\d}^\norigin=\psi_{\d\rd}^\norigin		&\text{$\boxplus$-indecomp.}\\
		21		&=&\psi_{\r\u}^\norigin=\psi_{\u\ru}^\norigin			&\text{$\boxplus$-indecomp.}\\[2pt]
		\hline\\[-10pt]
		123		&=&12\boxplus 12\\
		132		&=&\psi_{\r\d\rd}^\norigin=\psi_{\d\rd\u}^\norigin			&\text{$\boxplus$-indecomp.}\\
		213		&=&21\boxplus 12=\psi_{\d\rd\d}^\norigin\\
		231		&=&12\boxplus 21=\psi_{\u\ru\u}^\norigin\\
		312		&=&\psi_{\u\ru\d}^\norigin=\psi_{\r\u\ru}^\norigin			&\text{$\boxplus$-indecomp.}\\
		321		&=&21\boxplus 21\\[2pt]
		\hline\\[19.625pt]
		\end{array}
	&%
		\begin{array}[t]{rcll}
		\hline\\[-10pt]
		1234	&=&12\boxplus 12\boxplus 12\\
		1243	&=&12\boxplus 132\\
		1324	&=&132\boxplus 12\\
		1342	&=&\psi_{\r\d\rd\u}^\norigin	
				&\text{$\boxplus$-indecomp.}\\
		1423	&=&\psi_{\d\rd\u\ru}^\norigin
				&\text{$\boxplus$-indecomp.}\\
		1432	& &
				&\text{$\boxplus$-indecomp.}
			\\			
		2134	&=&21\boxplus 12\boxplus 12\\
		2143	&=&21\boxplus 132\\
		2314	&=&12\boxplus 21\boxplus 12\\
		2341	&=&12\boxplus 12\boxplus 21\\
		2413	&=&\psi_{\u\ru\u\ru}^\norigin=\psi_{\r\d\rd\d}^\norigin
				&\text{$\boxplus$-indecomp.}\\
		2431	&=&132\boxplus 21\\[2pt]
		\hline
		\end{array}
	\end{array}
\]}
\caption{Decompositions of permutations of lengths at most four (with only one of each symmetry class of length four represented). An initial letter $\r$ corresponds to either $\ru$ or $\rd$.}
\label{table-short-decomps}
\end{table}

Table~\ref{table-short-decomps} shows the decomposition of permutations of lengths at most four, as well as their expressions of the form of~$\psi_w^\norigin$, for those permutations that can be expressed that way. To cut the number of cases in half, when considering permutations of length four we utilize the fact that these concepts are invariant under \emph{complementation} of permutations (flipping their plots upside down).

From Table~\ref{table-short-decomps} we see that there are eight $\boxplus$-indecomposable permutations of length four ($1342$, $1423$, $1432$, $2413$, and their complements). In this table, an initial letter $\r$ corresponds to either $\ru$ or $\rd$, so of the permutations of length four, only $1423$ and its complement $4132$ have unique representations of the form~$\psi_w$, while the other six $\boxplus$-indecomposable permutations do not.

The permutation $1432$ and its complement $4123$ in fact cannot be expressed in terms of~$\boxplus$ and permutations of the form~$\psi^\norigin_w$ at all. This means that these two permutations are not subpermutations of any rightward-yearning pin sequence, and thus they do not arise in the classes we construct. Indeed, we have already seen this result, as $1432$ and $4123$ appear in the basis given in Theorem~\ref{thm-rightward-basis-enum}.

Our next result shows that all sufficiently long rightward-yearning pin sequences are $\boxplus$-indecomposable. The bound $|w|\ge 4$ is necessary because, as shown in Table~\ref{table-short-decomps},~$\psi_{\d\rd\d}^\norigin=213=21\boxplus 12$ and, symmetrically,~$\psi_{\u\ru\u}^\norigin=231=12\boxplus 21$.

\begin{lemma}\label{lem-pins-box-indecomposable}
Suppose that~$w\in\P$. If $|w|\ge 3$, then~$\psi_w^\origin$ is $\boxplus$-indecomposable. If $|w|\ge 4$, then~$\psi_w^\norigin$ is also $\boxplus$-indecomposable.
\end{lemma}

%

\begin{proof}
We proceed by induction on the length of~$w$. The base case for the claim about~$\psi_w^\norigin$ follows from examining Table~\ref{table-short-decomps} for~$\psi_w^\norigin$. For the claim about~$\psi_w^\origin$, the base case follows by the following computations:~$\psi^\origin_{\u\ru\u}=1342$, $\psi^\origin_{\u\ru\d}=\psi^\origin_{\rd\u\ru}=2413$, $\psi^\origin_{\ru\u\ru}=1423$, $\psi^\origin_{\ru\d\rd}=\psi^\origin_{\d\rd\u}=3142$, $\psi^\origin_{\d\rd\d}=4213$, and $\psi^\origin_{\rd\d\rd}=4132$.

Now let~$v$ denote the prefix of~$w$ comprising all but the last letter, and suppose by induction that~$\psi_v^\origin$ (resp.,~$\psi_v^\norigin$) is $\boxplus$-indecomposable.
The only way in which~$\psi_w^\norigin$ (resp.,~$\psi_{w}^\origin$) could be $\boxplus$-decomposable is if the last letter corresponds to a point that is inserted northeast or southeast of all of~$\psi_{v}^\norigin$ (resp.,~$\psi_{v}^\origin$), or if it is inserted next to the origin. Neither case is possible since pins must separate the predecessor pin from all the earlier ones. This means that the final pin cannot be placed in the top-right or bottom-right corner of~$\psi_w^\norigin$ (resp.\,$\psi_{w}^\origin$), and also, since $|\psi_v^\norigin|\ge 4$ (resp., $|\psi_v^\origin|\ge 4$), there are at least two other pins whose positions come between the origin and the final pin. Thus~$\psi_w^\norigin$ and~$\psi_w^\origin$ are $\boxplus$-indecomposable.
\end{proof}

Next, we investigate the uniqueness of the words encoding these $\boxplus$-indecomposable permutations.  In this direction, we are primarily interested in the uniqueness of words for permutations without an origin, but it is easier to first establish the result for permutations with an origin. Note that the bound $|w|\geq 4$ is best possible because~$\psi^\origin_{\u\ru\d}=\psi^\origin_{\rd\u\ru}=2413$.

\begin{table}
\[
	\begin{array}{c|c}
		w&\psi_w^\origin\\
		\hline\\[-10pt]
		\mathsf{ur_uur_u}&13524\\
		\mathsf{ur_udr_d}&35142\\
		\mathsf{r_uur_uu}&14253\\
		\mathsf{r_uur_ud}&25314\\
		\mathsf{r_udr_du}&31452\\
		\mathsf{r_udr_dd}&42513
	\end{array}
	\quad\quad\quad
	\begin{array}{c|c}
		w&\psi_w^\origin\\
		\hline\\[-10pt]
		\mathsf{dr_dur_u}&31524\\
		\mathsf{dr_ddr_d}&53142\\
		\mathsf{r_dur_uu}&24153\\
		\mathsf{r_dur_ud}&35214\\
		\mathsf{r_ddr_du}&41352\\
		\mathsf{r_ddr_dd}&52413
	\end{array}
\]
\caption{The 12 words of length 4 in $\P$ and their corresponding permutations~$\psi_w^\origin$.}
\label{table-short-words-with-origin}
\end{table}%

\begin{proposition}\label{prop-long-perms-have-unique-words}
If~$\psi_v^\origin=\psi_w^\origin$ for words $v,w\in\P$ satisfying $|v|=|w|\ge 4$, then $v=w$.
\end{proposition}%

\begin{proof}
We proceed by induction on $|v|=|w|$. The base case of $|v|=|w|=4$ follows from an examination of Table~\ref{table-short-words-with-origin}. Now suppose that $|v|=|w|\ge 5$. Note that the first entry of~$\psi_v^\origin$ must correspond to the origin~$p_0$, and every entry of~$\psi_v^\origin$ that lies above the origin must correspond to $\u$ or $\ru$, while every entry that lies below the origin must correspond to $\d$ or $\rd$.

Consider the rightmost entry of~$\psi_v^\origin$, which we may assume lies above the origin as the other case follows by a symmetrical argument. In this case, the last right step in both~$v$ and~$w$ is encoded by the letter $\ru$, and this letter is either the final or penultimate letter of both~$v$ and~$w$.

If $\ru$ is the final letter of both~$v$ and~$w$, then we can remove it---the permutations are still equal, and thus by induction, so are the words~$v$ and~$w$ with the last letter removed from each. Similarly, if $\ru$ is the penultimate letter in both~$v$ and~$w$, then the final letter of each word must be the same (corresponding to the second to last entry of~$\psi_v^\origin=\psi_w^\origin$), and again we can remove it and apply induction. In either case, we conclude that $v=w$.

It remains to consider the case where, without loss of generality, $\ru$ is both the final letter of~$v$ and the penultimate letter of~$w$. Thus we have, say, $v=x\ru$ and~$w=y\ru \ell$ where $x$ and $y$ are words of lengths at least four and three, respectively, and $\ell\in\{\u,\d\}$. Removing the rightmost point of~$\psi_v^\origin=\psi_{w}^\origin$ corresponds in each case to removing this final $\ru$. In the case of~$\psi_v^\origin$, this leaves us with~$\psi_x^\origin$, since $\ru$ was the last pin, whereas for~$\psi_{w}^\origin$ we obtain~$\psi_y^\origin\boxplus\psi_\ell^\origin$ since $\ru$ was an interior pin. However, we must have~$\psi_x^\origin=\psi_y^\origin\boxplus\psi_\ell^\origin$, and this is impossible since~$\psi_x^\origin$ is $\boxplus$-indecomposable by Lemma~\ref{lem-pins-box-indecomposable}, while~$\psi_y^\origin\boxplus\psi_\ell^\origin$ is not (note here that $\psi_\ell^\origin$ is of length two).
\end{proof}

The analogue to Proposition~\ref{prop-long-perms-have-unique-words} for pin sequences without an origin must account for the fact that~${\psi^\norigin_{\ru x}=\psi^\norigin_{\rd x}}$ for all words $x\in\P$. We begin, in fact, with a preliminary result that relates permutations with an origin to permutations without an origin. It would be possible to define the mapping~$\phi$ from Proposition~\ref{prop-origin-to-norigin} explicitly for shorter words as well, but we do not need to do so.

\begin{proposition}\label{prop-origin-to-norigin}
	If $\psi^\norigin_v=\psi^\origin_w$ for words $v,w\in\P$ satisfying $|v|=|w|+1\geq 5$, then $w=\phi(v)$, where~$\phi$ is defined as follows:
	\[\begin{array}{ccc}
	v&\quad&\phi(v)\\
	\hline
	\u\ru x&&\rd x\\
	\d\rd x&&\ru x\\
	\begin{aligned} \ru\u\ru x\\[-3pt] \rd\u\ru x\end{aligned}&\Big\}&\d\rd x\\
	\begin{aligned} \ru\d\rd x\\[-3pt] \rd\d\rd x\end{aligned}&\Big\}&\u\ru x
\end{array}\]
\end{proposition}

\begin{proof}
It is straightforward to verify that~$\psi_v^\norigin=\psi_{\phi(v)}^\origin$ in each case. Now suppose that $v,w\in\P$ have lengths at least~$5$ and~$4$, respectively, such that $\psi^\norigin_v=\psi^\origin_w$. By the above, we have $\psi^\origin_w=\psi^\origin_{\phi(v)}$, and hence by Proposition~\ref{prop-long-perms-have-unique-words} it follows that $w=\phi(v)$.
\end{proof}

\begin{proposition}\label{prop-long-perms-have-unique-words-norigin}
If~$\psi^\norigin_v=\psi^\norigin_w$ for words $v,w\in\P$ satisfying $|v|=|w|\ge 5$, then either~${v=w}$, or~${v=\ru x}$ and ${w=\rd x}$ for some word~$x\in\P$.
\end{proposition}

Note that~$\psi_{\u\ru\u\ru}^\norigin=\psi_{\rd\d\rd\d}^\norigin=2413$, so the restriction to words of length at least $5$ is best possible.

\begin{proof}
Consider words $v,w\in\P$ of length $\ge 5$ such that~$\psi^\norigin_v=\psi^\norigin_w$. By Proposition~\ref{prop-origin-to-norigin}, we have 
\[
	\psi^\origin_{\phi(v)}=\psi^\norigin_v=\psi^\norigin_w=\psi^\origin_{\phi(w)},
\]
and since $|\phi(v)|=|v|-1 \ge 4$, Proposition~\ref{prop-long-perms-have-unique-words} implies that~$\phi(v)=\phi(w)$. By the definition of~$\phi$, it now follows that either $v=w$, or~$v$ and~$w$ differ only in their first letter, as required. 
\end{proof}

The combination of Theorem~\ref{thm-Csk-description} with Propositions~\ref{prop-long-perms-have-unique-words} and~\ref{prop-long-perms-have-unique-words-norigin} provides us with a guarantee that any permutation~$\pi\in\C^{\exparen{s_k}}$ corresponds to an almost-unique collection of words $v_1$, $v_2$, $\dots$, $v_\ell\in\P^{\exparen{s_k}}$, providing $|v_1|\ge 5$ and $|v_i|\ge 4$ for $i\ge 2$. We finish this section by recording the precise classification for $\boxplus$-indecomposable permutations.

\begin{proposition}\label{prop-boxplus-indecomp}
A subpermutation~$\pi$ of a rightward-yearning pin sequence is $\boxplus$-indecomposable if and only if~$\pi\in\{1,12,21,132,312\}$ or~$\pi=\psi_w^\norigin$ for some word~$w\in\P$ with $|w|\ge 4$.
\end{proposition}

\begin{proof}
For permutations of lengths at most three, the result follows from an examination of Table~\ref{table-short-decomps}. Now suppose that~$\pi$ is a subpermutation of a rightward-yearning pin sequence and $|\pi|\ge 4$. If~$\pi=\psi^\norigin_w$ for some word~$w\in\P$, then $|w|\ge 4$, so Lemma~\ref{lem-pins-box-indecomposable} shows that~$\pi$ is $\boxplus$-indecomposable. Conversely, suppose that~$\pi$ is $\boxplus$-indecomposable. Since~$\pi$ is a subpermutation of a rightward-yearning pin sequence, Proposition~\ref{prop-boxplus-decomp} shows that there are nonempty words~$w_1$, $\dots$,~$w_k\in\P$ such that
\[
	\pi=
	\psi_{w_1}^\norigin\boxplus\psi_{w_2}^\origin\boxplus\cdots\boxplus\psi_{w_k}^\origin.
\]
Obviously, the only way that~$\pi$ could be $\boxplus$-indecomposable in this case would be if $k=1$. That is, $\pi=\psi_{w_1}^\norigin$, and thus $|w_1|=|\pi|\ge 4$.
\end{proof}

Finally, let us note that for any word $w\in\P$ with $|w|\geq 5$, Propositions~\ref{prop-long-perms-have-unique-words}---\ref{prop-long-perms-have-unique-words-norigin} show that there are at most three different expressions for $\psi_w^\norigin$, namely 
\[
\psi_w^\norigin,\quad \psi_{w'}^\norigin,\quad \psi_{x}^\norigin\boxplus\psi_{\phi(w)}^\origin = 1\boxplus \psi_{\phi(w)}^\origin=\psi_{\phi(w)}^\origin 
\]
where~$w$ and $w'$ satisfy $\phi(w)=\phi(w')$ (and thus differ in at most their first letter), and $x$ is any letter in $\alphabet$.

\section{Well-Quasi-Order}
\label{sec-wqo}

The fact that the classes $\C^{\exparen{s_k}}$ are, in a sense, ``$\boxplus$-closed'' enables us to establish that these classes are all wqo, although this takes a bit of preparation. Recall that the binary languages~$\L^{\exparen{s_k}}$ are wqo under the factor order by Proposition~\ref{prop-Lsk-wqo}. First we must lift this property to our pin sequence languages $\P^{\exparen{s_k}}$.

\begin{proposition}
\label{prop-Psk-wqo}
For every sequence $(s_k)$ of positive integers, the set $\P^{\exparen{s_k}}\subseteq\alphabet^\ast$ of words is wqo under the factor order.
\end{proposition}
\begin{proof}
Given any word~$w$, we define $\Delta_L(w)$ to be the word obtained by removing the first letter of~$w$ (assuming that~$w$ is nonempty). If~$u$ is contained as a factor in~$w$, it follows that $\Delta_L(u)$ is contained as a factor in $\Delta_L(w)$. Therefore, if the language $\mathcal{S}$ is wqo under the factor order, then the language $\Delta_L(\mathcal{S})$ is wqo under the factor order. We similarly define $\Delta_R(w)$ to be the word obtained by removing the last letter of~$w$.

By definition, for every word $v\in\P^{\exparen{s_k}}$, there is some word~$w\in\L^{\exparen{s_k}}$ for which~$v$ is a factor of $\rho(w)$. Indeed, if we take~$w$ to be minimal, then~$v$ comprises all but possibly the first and last letter of $\rho(w)$. Thus we can express $\P^{\exparen{s_k}}$ as
\[
	\P^{\exparen{s_k}}
	=
	\rho(\L^{\exparen{s_k}})
	\cup\Delta_L(\rho(\L^{\exparen{s_k}}))
	\cup\Delta_R(\rho(\L^{\exparen{s_k}}))
	\cup\Delta_L(\Delta_R(\rho(\L^{\exparen{s_k}}))).
\]
This shows that $\P^{\exparen{s_k}}$ is the union of four wqo posets, and is therefore wqo itself.
\end{proof}

To go from the languages $\P^{\exparen{s_k}}\subseteq\alphabet^\ast$ to the permutation classes $\C^{\exparen{s_k}}$, we need to first recall the setting and statement of Higman's lemma. Given a poset $(X,\le)$, we denote by $X^\ast$ the set (or language) of all words with letters from $X$. The \emph{generalized subword order} on $X^\ast$ is defined by stipulating that the word $v=v(1)\cdots v(k)$ is contained in the word~$w=w(1)\cdots w(n)$ if and only if~$w$ has a subsequence~$w({i_1})w({i_2})\cdots w({i_k})$ such that $v(j)\le w({i_j})$ for all indices $j$. The following is a weakened version of Higman's original result.

\newtheorem*{higmans-lemma}{\rm \textbf{Higman's lemma}~\cite{higman:ordering-by-div:}}
\begin{higmans-lemma}
If $(X,\le)$ is wqo, then $X^\ast$ is also wqo, under the generalized subword order.
\end{higmans-lemma}

Higman's lemma immediately implies (via Proposition~\ref{prop-Psk-wqo}) that the poset $(\P^{\exparen{s_k}})^\ast$ is wqo under the generalized subword order. Note that in this poset, the ``letters'' of a ``word'' are in fact words from~$\P^{\exparen{s_k}}$, which are themselves defined over the alphabet $\alphabet$. Thus an element $w\in(\P^{\exparen{s_k}})^\ast$ is a word of the form $w=w(1)w(2)\cdots w(\ell)$ for some $\ell\ge 0$ and words $w(1),w(2),\dots,w(\ell)\in\P^{\exparen{s_k}}$.

Now define a mapping~$\Phi : (\P^{\exparen{s_k}})^\ast\to\C^{\exparen{s_k}}$ by
\[
	\Phi(w(1)w(2)\cdots w(\ell))
	=
	\psi_{w(1)}^\norigin\boxplus\psi_{w(2)}^\origin\boxplus\cdots\boxplus\psi_{w(\ell)}^\origin.
\]
Proposition~\ref{prop-Psk-to-Csk} shows that the mappings~$w(i)\mapsto\psi_{w(i)}^\norigin$ and~$w(i)\mapsto\psi_{w(i)}^\origin$ from~$\P^{\exparen{s_k}}$ to~$\C^{\exparen{s_k}}$ are both order-preserving. It then follows from the definition of $\boxplus$ that~$\Phi$ is order-preserving. Theorem~\ref{thm-Csk-description} further implies that~$\Phi$ maps surjectively onto $\C^{\exparen{s_k}}$.

The main result of this section then follows from the general fact that if the domain of an order-preserving mapping is wqo, then its range must be as well. (This fact is easily proved by contradiction, for one could pull back any infinite antichain in the range of such a mapping to find an infinite antichain in its domain.)

\begin{proposition}
\label{prop-Csk-wqo}
For every sequence $(s_k)$ of positive integers, the permutation class $\C^{\exparen{s_k}}$ is wqo.
\end{proposition}

\section{Distinct Enumerations}
\label{sec-enum}

Having shown that the classes $\C^{\exparen{s_k}}$ are all wqo, we now finish the proof of Theorem~\ref{thm-wqo-not-algebraic} by establishing that these classes have distinct enumerations.

\begin{theorem}
\label{thm-distinct-enum}
	Suppose that $(s_k)$ and $(t_k)$ are distinct sequences of positive integers, and that $(s_k)$ lexicographically precedes $(t_k)$. Then there exists an integer $N$ such that 
	\[
		\bigcup_{n\le N}\C_n^{\exparen{t_k}} \subsetneq \bigcup_{n\le N}\C_n^{\exparen{s_k}}.
	\]
	In particular, the classes $\C^{\exparen{s_k}}$ and $\C^{\exparen{t_k}}$ have distinct enumeration sequences. 
\end{theorem}

\begin{proof}
Let $M$ be the integer from Proposition~\ref{prop-words-distinct-enum}, so~$\L^{\exparen{s_k}}_n = \L^{\exparen{t_k}}_n$ for all $n<M$, but~$\L^{\exparen{t_k}}_M\subsetneq \L^{\exparen{s_k}}_M$. We claim that $N=2M$ satisfies the requirements of the theorem. Since $M\ge 3$, we have $N\ge 6$.

First, we show that $\C^{\exparen{t_k}}_n\subseteq\C^{\exparen{s_k}}_n$ for all $n\leq N$. Let~$\pi\in\C^{\exparen{t_k}}$ be a permutation of length $n\leq N$.  By Theorem~\ref{thm-Csk-description}, we have
\[	\pi=\psi^\norigin_{v_1}\boxplus\psi^\origin_{v_2}\boxplus\cdots\boxplus\psi^\origin_{v_\ell}
\] 
for words $v_1$, $v_2$, $\dots$, $v_\ell \in\P^{\exparen{t_k}}$ with $|v_1|+\cdots +|v_\ell|=n$.

Suppose first that either $n<N$ or $\ell \ge 2$. For each $v_i$ (for $i=1,2,\dots,\ell$), take a shortest (binary) word~${w_i\in\L^{\exparen{t_k}}}$ such that~$v_i$ is a factor of $\rho(w_i)$, where $\rho : \{0,1\}^\ast\to\alphabet^\ast$ is the mapping defined in Section~\ref{sec-rightward-yearning}. Since~$v_i$ can have at most two fewer letters than~$\rho(w_i)$, we have $2|w_i|= |\rho(w_i)| \le |v_i|+2$. Furthermore, since $\sum|v_i|=n\leq N=2M$ and either $n<N$ or $\ell \ge 2$, we conclude that $|w_i| < M+1$ for each $i$ and hence~$w_i\in \L^{\exparen{s_k}}$ by Proposition~\ref{prop-words-distinct-enum}. Thus~${\rho(w_i)\in \P^{\exparen{s_k}}}$, and so~${v_i\in\P^{\exparen{s_k}}}$ for all~$i$ (since~$v_i$ is a factor of $\rho(w_i)$). Hence~$\pi\in\C^{\exparen{s_k}}$ by Theorem~\ref{thm-Csk-description} whenever $n<N$ or $\ell \ge 2$.

The proof of this half of the theorem is now reduced to considering a single $\boxplus$-indecomposable permutation of length $N$. Consider such a permutation, $\pi\in\C^{\exparen{t_k}}_N$, so $\pi=\psi_{v}^\norigin$ for a word~${v\in\P^{\exparen{t_k}}}$ of length~$M$. Since $\L^{\exparen{t_k}}_M\subsetneq \L^{\exparen{s_k}}_M$, we are done if $v=\rho(a)$ for some binary word $a\in\L^{\exparen{t_k}}_M$, since then we would have~$a\in\L^{\exparen{s_k}}_M$, $\rho(a)\in\P^{\exparen{s_k}}$, and consequently, $\pi\in\C^{\exparen{s_k}}_N$. Thus (by parity considerations), we need only consider the case where there exists a binary word $b\in\L^{\exparen{t_k}}$ of length~${M+1}$ such that~$v$ is formed by removing the first and last letters of~$\rho(b)$. This means in particular that~$v$ begins with~$\ru$ or~$\rd$. Without loss of generality suppose that~${v=\ru x}$ for some $x\in\P^{\exparen{t_k}}$. Since $|x|=2M-1$, we have $x\in\P^{\exparen{s_k}}$. Furthermore, there must exist a word of length $2M$ in $\P^{\exparen{s_k}}$ that contains $x$ as a suffix. Thus we must have either $\ru x\in \P^{\exparen{s_k}}$ or $\rd x\in \P^{\exparen{s_k}}$. Since $\pi=\psi_v^\norigin=\psi_{\ru x}^\norigin=\psi_{\rd x}^\norigin$, in either case we have $\pi\in\C^{\exparen{s_k}}_N$ and we are done.

Having established that $\C^{\exparen{t_k}}_n\subseteq\C^{\exparen{s_k}}_n$ for all $n\leq N$, we prove the second half of the theorem by exhibiting a permutation in $\C^{\exparen{s_k}}_N\setminus\C^{\exparen{t_k}}_N$.

First, fix a word~${w\in \L^{\exparen{s_k}}_M\setminus\L^{\exparen{t_k}}_M}$, which must exist because $\L^{\exparen{t_k}}_M\subsetneq \L^{\exparen{s_k}}_M$. The word $\rho(w)$ begins with $\u\ru$ or $\d\rd$, and also ends with one of these two pairs of letters. In any case, let $x$ be the word of length $2M-2\geq 4$ obtained by removing the first and last letters of~$\rho(w)$. We must have $x\notin\P^{\exparen{t_k}}$: otherwise, $x$ would be a factor of a word of the form $\rho(u)$ where $u\in\L^{\exparen{t_k}}$, and thus since the first and last letters of $\rho(w)$ are uniquely determined by the first and last letters of~$x$,~$\rho(w)$ would be a factor of $\rho(u)$, which would mean that~$w$ is a factor of~$u$, but that would imply that $w\in\L^{\exparen{t_k}}$, which we know is not the case.

Define the permutation
\[
	\pi=21\boxplus\psi_x^{\origin}.
\]
By construction, $\pi$ has length $N$. We claim that $\pi\in\C^{\exparen{s_k}}\setminus\C^{\exparen{t_k}}$. (While not necessary for the proof, it is worth remarking that even though $x\notin\P^{\exparen{t_k}}$, Proposition~\ref{prop-origin-to-norigin} shows that is possible to have $\psi_x^{\origin}=\psi_y^{\norigin}$ for some $y\in\P^{\exparen{t_k}}$, and thus we could have $\psi_x^{\origin}\in\C^{\exparen{t_k}}$; this is the reason we consider $21\boxplus\psi_x^{\origin}$ and not simply $\psi_x^{\origin}$.)

Note that we can write $\pi=\psi_{\u\ru}^\norigin \boxplus \psi_x^{\origin}$. Clearly $\u\ru\in\P^{\exparen{s_k}}$, and $x\in\P^{\exparen{s_k}}$ by its selection, so~${\pi\in\P^{\exparen{s_k}}}$ by Theorem~\ref{thm-Csk-description}.

Suppose, to the contrary, that we had~${\pi\in\P^{\exparen{t_k}}}$. By Theorem~\ref{thm-Csk-description}, this would mean that
\[
	\pi=\psi^\norigin_{v_1}\boxplus\psi^\origin_{v_2}\boxplus\cdots\boxplus\psi^\origin_{v_\ell}
\]
for words $v_1$, $v_2$, $\dots$, $v_\ell\in\P^{\exparen{t_k}}$. Lemma~\ref{lem-pins-box-indecomposable} shows that~$\psi_x^{\origin}$ is $\boxplus$-indecomposable, so the expression above must simplify to show that
\(
	\pi=21\boxplus\psi^\origin_{v_\ell}.
\)
However, Proposition~\ref{prop-long-perms-have-unique-words} states that we can only have~${\psi^\origin_{v_\ell}=\psi_x^{\origin}}$ if~${v_\ell=x}$, since $|x|\ge 4$. Thus, as $x\notin\P^{\exparen{t_k}}$, this is impossible, completing the proof that $\pi\notin\C^{\exparen{t_k}}$, and thus also the proof of the theorem.
\end{proof}

\section{Concluding Remarks}\label{sec-concluding}

\noindent{\bf Characterizing strongly algebraic classes}
Our main result, Theorem~\ref{thm-wqo-not-algebraic}, disproves False Conjecture~\ref{conj-wqo-algebraic} by showing that there are too many different enumerations of wqo permutation classes for them all to have algebraic generating functions. However, the weaker Questions~\ref{ques-lwqo-algebraic} and~\ref{ques-wqo-fb-algebraic}, which ask whether every lwqo or every finitely-based wqo permutation class has an algebraic generating function, might still be true.

Unfortunately, as there are strongly algebraic wqo-but-not-lwqo permutation classes,%
\footnote{One example of a strongly algebraic (indeed, strongly rational) permutation class that is wqo but not lwqo is the class $\Av(321,2341,3412,4123)$; see~\cite[Propositions~1.17 and~6.3]{brignall:labelled-well-q:} for the wqo and non-lwqo claims, and~\cite[Theorem~1.1]{albert:rationality-for:} for strong rationality of every wqo subclass of $\Av(321)$, including this example.}
a positive resolution to Question~\ref{ques-lwqo-algebraic} would not give a complete characterization of strongly algebraic classes, as a positive answer to False Conjecture~\ref{conj-wqo-algebraic} would have.

However, there remains a slim possibility of a tidy characterization of the strongly algebraic classes, via Question~\ref{ques-wqo-fb-algebraic}, which asks if every finitely-based wqo permutation class have an algebraic generating function. We suspect the answer to that question is negative, but if it was positive, it would imply that every finitely-based wqo permutation class was strongly algebraic (since every subclass of a finitely-based wqo class is itself finitely-based and wqo). In the other direction, every strongly algebraic class must be wqo (by elementary counting arguments), and we are not aware of a strongly algebraic class that is infinitely based.

\begin{question}
\label{ques-strong-alg-fb}
Is every strongly algebraic permutation class finitely based?
\end{question}

We also suspect the answer to Question~\ref{ques-strong-alg-fb} is also negative; that is, that there exists an infinitely based strongly algebraic permutation class, although we are not aware of such an example. Nevertheless, if the answers to Questions~\ref{ques-wqo-fb-algebraic} and~\ref{ques-strong-alg-fb} were both positive, it would follow that a class is strongly algebraic if and only if it is both wqo and finitely based.

\medskip\noindent{\bf Growth rates}
For the remainder of the conclusion we discuss a rougher notion of enumeration, that of growth rates. The \emph{upper growth rate} of the permutation class~$\C$ is defined by
\[
	\ugr(\C) = \limsup_{n\to\infty}\sqrt[n]{|\C_n|}.
\]
When the limit superior above is actually a limit (as has been conjectured is always the case%
\footnote{A weaker version of this conjecture, applying only to finitely-based permutation classes, was explicitly asked by the second author at the conference \emph{Permutation Patterns 2005}~\cite[Section~4]{elder:problems-and-co:}, but that may not be the first time it was posed. A stronger conjecture, stated in the context of ordered graphs, has been posed by Balogh, Bollob\'as, and Morris~\cite[Conjecture 8.1]{balogh:hereditary-prop:ordgraphs}.})
then we call this the \emph{proper growth rate} of the class and denote it by $\gr(\C)$.

There has been extensive research on the set of all possible growth rates of permutation classes and at what growth rates certain phenomena first appear~\cite{albert:growing-at-a-pe:,bevan:intervals-of-pe:,huczynska:grid-classes-an:,kaiser:on-growth-rates:,pantone:growth-rates-of:,vatter:permutation-cla:lambda:,vatter:small-permutati:,vatter:growth-rates-of:}. For example, the smallest non-wqo permutation class (equivalently, the infinite antichain whose downward closure has the smallest possible growth rate) has growth rate denoted by $\kappa$, the largest real root of $x^3-2x^2-1$, which is approximately $2.21$~\cite{vatter:small-permutati:}. Furthermore, it has been established that every permutation class of growth rate smaller than $\kappa$ has a rational generating function~\cite[Section~8]{albert:inflations-of-g:}. Thus it may be interesting to determine the growth rates of the classes $\C^{\exparen{s_k}}$, or perhaps more ambitiously, to answer the following question.

\begin{question}
What is the smallest real number $\omega$ such that there is a wqo permutation class of growth rate $\omega$ that fails to have an algebraic generating function?
\end{question}

It is not difficult to show that the classes $\C^{\exparen{s_k}}$ have \emph{proper} growth rates, as we now explain. This follows from a Fekete's lemma argument that seems to have first been stated in the permutation context by Arratia~\cite{arratia:on-the-stanley-:}.%
\footnote{Although it seems likely given his some of his writing~\cite[Problem~28]{chvatal:selected-combin:} that Knuth was aware of this application of Fekete's lemma in 1972.}
We use the multiplicative form of Fekete's lemma, which states that if the sequence~$(a_n)$ of positive real numbers satisfies $a_{m+n}\ge a_ma_n$ for all~$m$ and~$n$ (that is, if it is supermultiplicative), then~$\lim\sqrt[n]{a_n}$ exists and equals $\sup\sqrt[n]{a_n}$.

Note that the following result has since been generalized by Jarvis~\cite{jarvis:pin-classes:}.

\begin{proposition}
For every sequence $(s_k)$ of positive integers, the permutation class $\C^{\exparen{s_k}}$ has a growth rate.	
\end{proposition}

\begin{proof}
Let~$\pi\in\C^{\exparen{s_k}}_m$ and~$\sigma\in\C^{\exparen{s_k}}_n$. By Theorem~\ref{thm-Csk-description},~$\sigma$ can be expressed as
\[
	\sigma
	=
		\psi^\norigin_{w_1}\boxplus\psi^\origin_{w_2}\boxplus\cdots\boxplus\psi^\origin_{w_\ell}
\]
for words~$w_1$, $\dots$,~$w_\ell\in\P^{\exparen{s_k}}$. The mapping that sends the pair $(\pi,\sigma)$ to the permutation
\[
	\pi\boxplus\psi^\origin_{w_1}\boxplus\psi^\origin_{w_2}\boxplus\cdots\boxplus\psi^\origin_{w_\ell}\in\C^{\exparen{s_k}}_{m+n}
\]
is an injection from $\C^{\exparen{s_k}}_m\times \C^{\exparen{s_k}}_n$ to $\C^{\exparen{s_k}}_{m+n}$, and thus~$\lim\sqrt[n]{|\C^{\exparen{s_k}}_n|}$ exists by Fekete's lemma.
\end{proof}

We leave open the question of what the growth rates of the classes $\C^{\exparen{s_k}}$ actually are, and whether these growth rates even depend on the choice of the sequence $(s_k)$.

\minisec{Acknowledgements}
We are grateful to Jay Pantone for generating the heatmaps displayed in Figure~\ref{fig-heatmap-1432-4123}.

\setlength{\bibsep}{4pt}

\begin{small}
\bibliographystyle{acm}
\bibliography{Pouzet}
\end{small}

\end{document}